\documentclass[a4paper, 12pt, oneside]{amsart}

\usepackage{amsfonts,amsmath,amssymb,amsxtra,amsthm}
\usepackage{mathrsfs,mathalfa}
\usepackage{times}
\usepackage{anysize}
\usepackage{graphicx}
\usepackage{xcolor}

\input xy
\xyoption{all}
\usepackage[all,knot]{xy}

\newcommand{\Ab}{\operatorname{\bf{Ab}}}
\newcommand{\eff}{\operatorname{\bf{eff}}}

\newcommand{\op}{\operatorname{op}\nolimits}
\newcommand{\supp}{\operatorname{supp}\nolimits}

\newcommand{\AR}{\operatorname{AR}\nolimits}
\newcommand{\Ind}{\operatorname{Ind}\nolimits}

\newcommand{\add}{\operatorname{add}\nolimits}
\newcommand{\fp}{\operatorname{fp}\nolimits}

\newcommand{\Hom}{\operatorname{Hom}\nolimits}

\newcommand{\Imm}{\operatorname{Im}\nolimits}

\newcommand{\coker}{\operatorname{Coker}\nolimits}
\newcommand{\Ext}{\operatorname{Ext}\nolimits}
\newcommand{\Ex}{\operatorname{Ex}\nolimits}

\newcommand{\Iso}{\operatorname{Iso}\nolimits}
\newcommand{\Mod}{\operatorname{Mod }\nolimits}
\renewcommand{\mod}{\operatorname{mod}\nolimits}

\newcommand{\rep}{\operatorname{rep}\nolimits}

\newcommand{\ind}{\operatorname{ind}\nolimits}

\newcommand{\defff}{\operatorname{def}\nolimits}

\newcommand{\arr}[1]{\stackrel{sharp1}{\rightarrow}}
\newcommand{\larr}[1]{\stackrel{sharp1}{\leftarrow}}

\newcommand{\ca}{{\mathcal A}}
\newcommand{\cA}{{\mathcal A}}
\newcommand{\cb}{{\mathcal B}}
\newcommand{\cc}{{\mathcal C}}
\newcommand{\cC}{{\mathcal C}}
\newcommand{\cd}{{\mathcal D}}
\newcommand{\cE}{{\mathcal E}}
\newcommand{\ce}{{\mathcal E}}
\newcommand{\cF}{{\mathcal F}}

\newcommand{\cH}{{\mathcal H}}
\newcommand{\ch}{{\mathcal H}}
\newcommand{\ci}{{\mathcal I}}
\newcommand{\cj}{{\mathcal J}}

\newcommand{\cP}{{\mathcal P}}
\newcommand{\cp}{{\mathcal P}}

\renewcommand{\tilde}[1]{\widetilde{sharp1}}

\newcommand{\ul}[1]{\underline{sharp1}}
\newcommand{\ol}[1]{\overline{sharp1}}
\renewcommand{\hat}[1]{\widehat{sharp1}}

\newtheorem{theorem}{Theorem}[section]
\newtheorem{corollary}[theorem]{Corollary}

\newtheorem{lemma}[theorem]{Lemma}
\newtheorem{proposition}[theorem]{Proposition}

\newtheorem{thmx}{Theorem}

\theoremstyle{definition}
\newtheorem{definition}[theorem]{Definition}
\newtheorem{example}[theorem]{Example}
\newtheorem{remark}[theorem]{Remark}

\begin{document}

\title{Exact structures and degeneration of Hall algebras}


\author{Xin Fang}
\address{Xin Fang:\newline
Abteilung Mathematik, Department Mathematik/Informatik, Universit\"at zu K\"oln, 50931, Cologne, Germany}
\email{xinfang.math@gmail.com}
\author{Mikhail Gorsky}
\address{Mikhail Gorsky:\newline
Instute of Algebra and Number Theory, University of Stuttgart, Pfaffenwaldring 57, 70569 Stuttgart, Germany}
\email{mikhail.gorsky@iaz.uni-stuttgart.de}

\maketitle

\begin{abstract}
We study degenerations of the Hall algebras of exact categories induced by degree functions on the set of isomorphism classes of indecomposable objects. We prove that each such degeneration of the Hall algebra $\ch(\ce)$ of an exact category $\ce$ is the Hall algebra of a smaller exact structure $\ce' < \ce$ on the same additive category $\ca.$ When $\ce$ is admissible in the sense of Enomoto, for any $\ce' < \ce$ satisfying suitable finiteness conditions, we prove that $\ch(\ce')$ is a degeneration of $\ch(\ce)$ of this kind.

In the additively finite case, all such degree functions form a simplicial cone whose face lattice reflects properties of the lattice of exact structures. For the categories of representations of Dynkin quivers, we recover degenerations of the negative part of the corresponding quantum group, as well as the associated polyhedral structure studied by Fourier, Reineke and the first author.

Along the way, we give minor improvements to certain results of Enomoto and Br\"ustle-Langford-Hassoun-Roy concerning the classification of exact structures on an additive category. We prove that for each idempotent complete additive category $\ca$, there exists an abelian category whose lattice of Serre subcategories is isomorphic to the lattice of exact structures on $\ca.$ 
We show that every Krull-Schmidt category admits a unique maximal admissible exact structure and that the lattice of smaller exact structures of an admissible exact structure is Boolean.
\end{abstract}

\section{Introduction}

\subsection{Hall algebras and quantum groups}
Hall algebras provide one of the first known examples of additive categorification. They first appeared in works of Steinitz \cite{St} and Hall \cite{Hal} on commutative finite $p$-groups. Later, they reappeared in the work of Ringel \cite{R1} on quantum groups. He introduced the notion of the Hall algebra of an abelian category with finite $\Hom-$ and $\Ext^1$-spaces. As a vector space, it has a basis parameterized by the isomorphism classes of objects in the category. The multiplication captures information about the extensions between objects. One can consider this as an algebra of constructible functions on the moduli stack of objects in the category, with the convolution product given by the Hecke correspondences.

Ringel constructed an isomorphism between the twisted Hall algebra of the category of representations of a Dynkin quiver $Q$ (i.e. a quiver of type {\tt ADE}) over the finite field $\mathbb{F}_q$ and the nilpotent part of the corresponding quantum group, specialized at the square root of $q:$ 
$$U_{\sqrt{q}}^-(\mathfrak{g}(Q)) \overset\sim\to \mathcal{H}_{tw}(\rep_{\mathbb{F}_q}(Q)).$$
Later Green \cite{Gr} generalized this result to an arbitrary valued quiver $Q$ by providing an isomorphism between the nilpotent part of the quantized universal enveloping algebra of the corresponding Kac-Moody algebra and the so-called ``composition'' subalgebra in $\mathcal{H}_{tw}(\rep_{\mathbb{F}_q}(Q))$ generated by the classes of simple objects. Using the Grothendieck group of the category of quiver representations, he introduced an extended version of the Hall algebra which recovers the Borel part of the quantum group. 
 
Lusztig \cite{Lusztig90, Lusztig91} investigated the geometric version of the composition subalgebra in the Hall algebra $\mathcal{H}_{tw}(\rep_{\mathbb{F}_q}(Q)),$ using perverse sheaves on moduli spaces of quiver representations. This is an example of monoidal categorification, where the tensor product in a certain monoidal category gives rise to the multiplication in the algebra. This approach led him to the discovery of the canonical basis in $U_{\sqrt{q}}^-(\mathfrak{g}(Q))$ satisfying very pleasant positivity properties. 

Other interesting examples of Hall algebras are those of categories of coherent sheaves on schemes. They were first considered by Kapranov in \cite{KaprCoh}, where he linked Hall algebras of coherent sheaves on curves to the study of automorphic forms. Since then, Hall algebras of coherent sheaves have been studied intensively and turned out to be related to the geometric Langlands conjecture, Cherednik algebras, knot invariants, etc, see \cite{SchII, SV, Sch18, GN, MS} and references therein.
Hall algebras of other classes of categories appeared in the context of skein modules of Legendrian links \cite{Haiden}, Donaldson-Thomas invariants (see the survey \cite{BrHallDT} and references therein) and cluster algebras \cite{BrScat}.

Hubery \cite{Hu} proved that the algebra defined in the same way as by Ringel, but for an exact category, is also unital and associative. Since any small exact category can be embedded as a full extension-closed subcategory into an abelian category, this might look like a rephrasing of Ringel's theorems. However, the result is deep for several reasons. First, for an arbitrary $\Hom-$ and $\Ext^1-$finite exact category it might be impossible to find an embedding (as a full extension-closed subcategory) into a $\Hom-$ and $\Ext^1-$finite abelian category \cite[Section 2.1]{BG}. Second, the associativity of such a Hall algebra corresponds to the fact that the Waldhausen $S_{\bullet}$-space of an exact category is {\it 2-Segal} in the sense of Dyckerhoff-Kapranov \cite{DK}. Waldhausen spaces first appeared in the studies of higher K-theory, where it is important to go beyond the realm of abelian categories. 

Our work sheds new light on Hubery's result.
Being abelian is a property of an additive category, while an exact category is an additive category endowed with an extra structure. The Hall algebra of an exact category depends not only on the underlying additive category, but on this structure as well. Moreover, since each additive category admits a split exact structure, to each $\Hom-$finite $\mathbb{F}_q-$linear additive category one can associate at least one Hall algebra. We study relations between different Hall algebras associated to different exact structures on the same additive category.

\subsection{Exact structures on an additive category}
An additive category $\cA$ can be endowed with many different exact structures. Some classification results of these exact structures can be found in the papers \cite{DReitenSS, EnochsJenda_v2, Rump11, Crivei12} and the references therein. Recently, two new approaches are proposed to this topic. There is a natural partial order on exact structures on an additive category: one structure is smaller than the other if it has less conflations. Enomoto \cite{Enomoto1} used the functorial approach to classify exact structures in a large class of additive categories in terms of Serre subcategories of their module categories. In particular, he proved that, in some natural generality, exact structures on an \emph{additively finite} category $\ca$ (that is, an additive category having finitely many indecomposable objects up to isomorphism)\footnote{Enomoto calls such $\cA$ addtitive categories \emph{of finite type}, but this term is used in literature for several different classes of categories and so may be a bit ambiguous.} form a Boolean lattice. This lattice can be identified with the lattice of the subsets of the set of all the indecomposable objects that are not projective with respect to the maximal exact structure on $\ca$. Br\"ustle, Hassoun, Langford and Roy investigated relations between changes of exact structures to smaller ones (i.e. having strictly less conflations) and some matrix reduction problems, see  \cite{BHLR} for the details (see also \cite{DReitenSS}). This inspired the name of the \emph{reduction of exact structures} for this procedure. They also gave a partial generalization of Enomoto's theorem by showing that exact structures on an additive category always form a complete bounded lattice. For any additive category, the minimal element is always the \emph{split}, or \emph{additive} exact structure, that we denote by $\ce^{\add}.$ The maximal element is denoted by $\cE^{\max},$ its existence was proved by Rump \cite{Rump11}.\footnote{One should always keep in mind that the structures $\ce^{\add}$ and $\ce^{\max}$ depend on the category $\ca$.}

To an exact structure $\ce$ on an additive category $\ca$, one can associate a category of \emph{effaceable functors} $\eff \ce$ (such functors go at least back to Grothendieck), that can be also interpreted as the category of Auslander's \emph{contravariant defects} of conflations. It was known for a long time that this category is a Serre subcategory (in fact, even a localizing subcategory) in the category $\Mod \ca$ of additive functors $\ca^{\op} \to \Ab$ \cite[Appendix A]{Kel1}, see also \cite{Fiorot}. Enomoto \cite{Enomoto1} classified all subcategories of $\Mod \ca$ that arise as categories of defects of exact structures on $\cA.$ By using some properties of these categories proved by Enomoto in \cite{Enomoto1,Enomoto2}, we slightly improve on his classification and provide a simpler description of the lattice of exact structures. 

\begin{thmx} [= Theorem \ref{classif_general:body}] \label{classif_general} 
The lattice of exact structures on an arbitrary idempotent complete additive category $\ca$ is isomorphic to the lattice of Serre subcategories of the abelian category $\eff (\ce^{\max}).$
\end{thmx}

For an exact category $\cE,$ one can investigate the structure of the interval $[\ce^{\add}, \ce]$ in this lattice. As an interval in a complete bounded lattice, it is itself a complete bounded lattice. It is isomorphic to the lattice of Serre subcategories of the category $\eff \ce$. In general, its structure can be quite complicated, but we show that in the following generality it is Boolean. Enomoto \cite{Enomoto1} introduced the notion of \emph{admissible} exact categories. He calls an exact category $\ce$ admissible if any object in the category $\eff \ce$ has finite length. In \cite{Enomoto2}, he further investigated the relation between the admissiblity of an exact category and the structure of its \emph{almost split}, or \emph{Auslander-Reiten} conflations. We recall the necessary definitions in Section \ref{AR+admissible}. We formulate and prove explicitly the following statement, again essentially by combining some results of Enomoto \cite{Enomoto1, Enomoto2}.

\begin{thmx}[=Theorem \ref{classif_adm:body}] \label{classif_adm}
Let $\cE$ be an admissible exact category on a Krull-Schmidt category $\cA.$ Then the following lattices are isomorphic:
\begin{itemize}
    \item The interval $[\ce^{\add}, \ce]$ in the lattice of exact structures;
    \item The lattice of subsets of the set of AR conflations in $\cE;$
    \item The lattice of subsets of the set of all the indecomposables in $\cA$ that are not projective in $\cE$.
\end{itemize}
In particular, the interval $[\ce^{\add}, \ce]$ is a Boolean lattice.
\end{thmx}

Note that the condition on the exact category to be admissible does not impose restrictions on the underlying additive category. The split exact structure on any additive category is admissible, and there are fairly natural examples of admissible non-split exact structures on additive categories, whose maximal exact structures are not admissible. We discuss this further in Section \ref{AR+admissible}. From Theorem \ref{classif_general}, we immediately deduce the following notable result.

\begin{thmx}[=Theorem \ref{max_adm:body}]
\label{maximal_admissible}
Every Krull-Schmidt category $\cA$ admits a unique maximal admissible exact structure, namely the structure $\cE$ whose category of defects $\eff \cE$ is the full subcategory of modules of finite length in $\eff (\ca, \cE^{\max})$. 
\end{thmx}

It follows from the work of Zhu and Zhuang \cite{ZZ} that the class of Krull-Schmidt categories $\cA$ whose maximal exact structures are admissible is much larger than that of additively finite categories. In particular, it includes all $\Hom-$finite, \emph{locally finite} categories. These are the categories where for each object $X$, both $\Hom(X,-)$ and $\Hom(-, X)$ vanish on all but finitely many indecomposables.  We give an alternative proof of this statement in Section \ref{AR+admissible}. 

For each full subcategory $\cb \subset \ca$ and each exact structure $\ce$, there exist a canonical pair of smaller exact structures $\ce^{\cb}, \ce_{\cb} \in [\ce^{\add}, \ce]$ where objects of $\cb$ are projective, resp. injective. In general, there are many structures $\ce' \in  [\ce^{\add}, \ce]$ not of this form. We prove that an exact structure $\ce$ is admissible if and only if each smaller exact structure has form $\ce^{\cb}$ for a certain full subcategory $\cb \subset \ca$ and form $\ce_{\cb'}$ for (another) full subcategory $\cb' \subset \ca$ (Proposition \ref{relative_subcat}).

\subsection{Relative homological algebra}

Given an exact structure $\ce$, considerations of smaller exact structures $\ce' < \ce$ can be seen from a different point of view. For each pair of objects $A, B \in \cA$ and each exact structure $\ce' < \ce,$ we have $\Ext^1_{\ce'}(A, B) \subset \Ext^1_{\ce}(A, B).$ In other words, $\Ext^1_{\ce'}(-, -)$ is a sub-bifunctor of $\Ext^1_{\ce}(-,-).$ Such sub-bifunctors are the principal subject of \emph{relative homological algebra}. This topic goes back to Hochschild \cite{Hochschild}. It was further developed by Butler and Horrocks \cite{ButlerHorrocks} and Auslander and Solberg \cite{AS}, see also the survey \cite{Solberg} for historical remarks and references. There is a certain class of additive sub-bifunctors of $\Ext^1_{\ce}(-,-)$, called \emph{closed}. It was proved in \cite{DReitenSS} that an additive sub-bifunctor of 
$\Ext^1_{\ce}(-,-)$ is closed if and only if it has the form $\Ext^1_{\ce'}(-, -)$ for some exact structure $\ce' < \ce.$ Our paper could have been thus called ``Relative homological algebra and degeneration of Hall algebras''. Theorem \ref{classif_general} is also a generalization of a result of Buan \cite{Buan} concerning closed additive sub-bifunctors of $\Ext^1_{\Lambda}(-,-)$ on categories of modules $\mod \Lambda$ over Artin algebras.

We would like to note that the existence of the maximal exact structure on an arbitrary additive category $\cA$ proved by Rump \cite{Rump11} implies that, more generally, studies of exact structures on $\ca$ are always the same as studies of closed additive sub-bifunctors of a certain bifunctor $\ca^{\op} \times \ce \to \Ab.$ This consequence of Rump's result seems to be underappreciated. However, given an exact structure $\ce$, it is often more reasonable to consider it by itself and not as a sub-structure of the maximal exact structure, as the latter may behave much worse.

\subsection{Exact structures and degenerations of Hall algebras}

It is a natural question to ask how the Hall algebras of different exact structures on the same additive category $\cA$ are related between each other. By using Theorem \ref{classif_general}, we answer this question. We consider functions $w: \Iso(\cA) \to \mathbb{N}.$ For an exact structure $\ce$ on $\ca$, some of these functions induce algebra filtrations on its Hall algebra $\ch(\ce).$ We call such functions \emph{$\cE-$quasi valuations}. When such a function is additive on direct sums, we call it an \emph{$\cE-$valuation}. 

\begin{thmx}[= Theorem \ref{each_degen_is_HA:body}]
Let $\cA$ be a $\Hom-$finite idempotent complete $\mathbb{F}_q-$linear additive category and let $\cE$ be an $\Ext^1-$finite exact structure on $\cA.$ Then for each $\cE-$valuation $w:~\Iso(\cA) \to \mathbb{N},$ the associated graded algebra induced by $w$ on the Hall algebra $\cH(\cE)$ is the Hall algebra $\cH(\cE'),$ for some $\cE' \in [\ce^{\add}, \ce].$
\end{thmx}

Given a pair of exact structures $\ce' < \ce$ on $\ca,$ we would like to study the $\ce-$valuations $w$ such that the degeneration of $\ch(\ce)$ induced by $w$ is precisely $\ch(\ce').$ In this work, we construct such $w$ explicitly for a pair of admissible exact structures satisfying suitable finiteness conditions.

\begin{thmx}[= Theorem \ref{degen:body}] \label{degen}
Let $\cA$ be a $\Hom-$finite $\mathbb{F}_q-$linear additive category. Suppose $\cA$ is endowed with two $\Ext^1-$finite admissible exact structures $\cE' < \cE.$ If $$\supp \Hom(-, M)|_{\cp(\ce') \backslash\cp(\ce)} < \infty, \quad \forall M \in \cA,$$
then $\cH(\cE')$ is the associated graded algebra of $\cH(\cE)$ with respect to a filtration given by an $\ce-$valuation defined by
\begin{align*}
w_{\cE, \cE'}(M) = \sum\limits_{P \in \ind(\cp(\ce'))\backslash \ind(\cp(\ce))}\dim\Hom_\cA(P, M).
\end{align*}
\end{thmx}

Here by $\supp \Hom(-, M)$ we mean the set of indecomposable objects $X$ such that $\Hom(X, M) \neq 0$. When $\cA$ is locally finite, thanks to the theorem of Zhu and Zhuang \cite{ZZ} mentioned above, Theorem \ref{degen} applies to all pairs of $\Ext^1-$finite exact structures.

The category of representations of a Dynkin quiver is additively finite. This implies that the maximal (i.e. abelian) exact structure on it is admissible.
In this case, we prove that the degenerations of the Hall algebra of the abelian category given by smaller exact structures can be identified precisely with the degenerations defined in \cite{FFR1}. In particular, the split exact structure corresponds to a $q-$commutative polynomial algebra, this recovers the famous Poincar\'e-Birkhoff-Witt (PBW) theorem in this quantized setting (Corollary \ref{Cor:q-symm}). Thus, Theorem \ref{degen} can be thought of as a generalization of the PBW theorem in the framework of Hall algebras. Our work is not the first such generalization. Namely, Berenstein and Greenstein \cite{BG} proved a PBW-type theorem for all finitary exact categories. We discuss the difference in our methods and results in Section \ref{BG}.

Our proof of Theorem \ref{degen} uses tools of the Auslander-Reiten theory, namely the fact that in any admissible exact category, each conflation can be seen as a \emph{non-negative} linear combination of AR-conflations. This property does not hold for arbitrary exact structures on additive categories, see \cite{Enomoto2} for details and references. It was first studied by Auslander \cite{Aus4} and Butler \cite{Butler} for the categories of representations of finite-dimensional algebras (endowed with the abelian exact structure). It should be stressed out that while this property is usually formulated in terms of subgroups in $K_0(\ce^{\add})$ and denoted by ``$\Ex(\ce) = \AR(\ce)$'', it actually concerns \emph{semigroups} (see already Auslander's proof in \cite{Aus4}), and so should be rather denoted by ``$\Ex_+(\ce) = \AR_+(\ce)$''. We discuss this in Section \ref{AR+admissible}.

The theorem of Berenstein-Greenstein shows that $\ch(\ce^{\add})$ is a degeneration of $\ch(\ce)$ even beyond the admissible case. It is possible that this degeneration can be obtained from an $\ce-$valuation. Moreover, we believe that for $\ce' < \ce$ one can find $\ce-$valuations inducing degenerations from $\ch(\ce)$ to $\ch(\ce')$ in much larger generality than that of Theorem \ref{degen}. In general, it will not be enough to take formal sums of functions of the form $\dim \Hom(X, -)$ for objects $X \in \cA.$ While we have a conjectural approach to this problem, a lot of details are to be worked out. We plan to address it in future work.

\subsection{Polyhedral cones}
The Hall algebra $\cH(\ce)$ is defined by designating a multiplication on a linear basis indexed by the objects in $\ce$. When the category $\ce$ is furthermore Krull-Schmidt and additively finite, then it admits a presentation by finitely many generators and relations. 

In Gr\"obner theory, one associates to an ideal $I$ in a commutative polynomial ring $R$ a polyhedral fan, called the Gr\"obner fan of the ideal. Each maximal dimensional cone in this fan gives the quotient ring $R/I$ a linear basis consisting of standard monomials with respect to this cone; the multiplication of elements in such a basis endows $R/I$ with a structure termed "algebra with a straightening law" (ASL).   

The first cone $\mathcal{D}$ we will associate to the Hall algebra $\cH(\ce)$ is motivated by this construction in Gr\"obner theory, but in a non-commutative setting. Such a cone parameterizes all possible non-negative degrees that can be assigned to a natural generating set of $\cH(\ce)$ consisting of the class of indecomposable objects, so that they induce algebra filtrations on $\cH(\ce)$. For a point in $\mathcal{D}$, by taking the associated graded algebra one gets a degeneration of $\cH(\ce)$. If such a point comes from the interior of $\mathcal{D}$, the degenerate algebra is isomorphic to a skew-polynomial algebra.

The second cone $\mathcal{C}$ arises from the following observation in the degeneration of modules: if a module $M$ degenerates to $N$, then for any test module $T$, $\dim\Hom(T,M)\leq \dim\Hom(T,N)$. Semi-simplification is a particular case of degeneration, hence any module $T$ defines via $\dim\Hom(T,-)$ a filtration on the Hall algebra hence a point in the cone $\mathcal{D}$. 

Motivated by the Poincar\'e-Birkhoff-Witt filtration in Lie algebras, this point of view was taken in \cite{FFR1, FFR2} for the category consisting of finitely generated modules of the path algebra of a Dynkin quiver taken with the abelian exact structure, where it is shown that the cones $\mathcal{C}$ and $\mathcal{D}$ are dual cones via $\dim\Hom(-,-)$. As a consequence, $\mathcal{D}$ is a non-empty polyhedral cone, which is simplicial and the equations of its facets are known.

In this paper, we generalize this result to the framework of exact categories. In fact, we study a relative version of this result involving two exact structures $\cE' < \cE$ on a $\Hom-$finite idempotent complete additively finite category. In this situation, we define two cones:
\begin{enumerate}
    \item the cone $\mathcal{D}^{\cE,\cE'}$ consisting of degree functions which can be imposed on the indecomposables in $\mathcal{A}$ such that they induce algebra filtrations on $\mathcal{H}(\mathcal{E})$ whose associated graded algebras are $\mathcal{H}(\mathcal{E'})$; 
    \item the cone $\mathcal{C}^{\cE,\cE'}$, lying in the kernel of $K_0(\cE')\to K_0(\cE)$, and generated by the differences $[A]-[B]+[C]$ for all conflations $A \rightarrowtail B \twoheadrightarrow C$ in $\cE\backslash\cE'$. 
\end{enumerate}    
We will show that these cones are dual to each other:

\begin{thmx}[=Theorem \ref{Thm:cone}]
We have
$$\mathcal{D}^{\cE,\cE'}=\{\varphi\in (K_0(\cE^{\mathrm{add}})\otimes_\mathbb{Z}\mathbb{R})^*\mid \text{for any }x\in\cC^{\cE,\cE'}, \varphi(x)>0;\ \text{for any }y\in\cC^{\cE',\cE^{\mathrm{add}}},\ \varphi(y)=0\}.$$
\end{thmx}

These cones are simplicial. By taking the face lattice, we obtain a polyhedral geometric interpretation of the 
Boolean property of the interval $[\ce', \ce]$ in the lattice of exact structures (Remark \ref{Rmk:Enomoto}).

Some examples arising from quiver representations are discussed in Section \ref{Examples}. The principal example is the category of representations of an equioriented quiver of type $\tt A$. In this situation, there exists a face of $\mathcal{D}$ standing at the crossroads of representation of algebras,  Lie theory and combinatorial commutative algebra (Example \ref{Ex:An}).
We also discuss the category of representations of a quiver of type $\tt A_3$ endowed with two non-trivial exact structures (Example \ref{example:A3}) and of the disjoint union of two quivers of type $\tt A_2$ (Example \ref{disjoint_A2}).

\subsection{Extriangulated categories}

Let us note that Zhu and Zhuang \cite{ZZ} actually work in the setting of {\it extriangulated categories}. This notion was recently introduced by Nakaoka and Palu \cite{NakaokaPalu1} as a unification of exact and of triangulated categories. Under suitable finiteness conditions, we can generalize our main results to the setting of extriangulated categories. Since it requires to present the rather long definition of extriangulated categories, to define Hall algebras of those, and to give a quite technical proof of their associativity (that is implicitly suggested by the work of Dyckerhoff and Kapranov \cite{DK} combined with that of Nakaoka and Palu \cite{NakaokaPalu2}), we prefer to postpone it to the upcoming sequel to this paper. We just give a brief remark in Section \ref{extriangulated}.

\subsection{Structure of the paper}
In Section \ref{Sec:2}, we collect basics on exact categories and give a review of the structure of AR-conflations in an admissible exact category. We study the lattice of exact structures on an additive category. The definition of the Hall algebra associated to an exact category is briefly recalled in Section \ref{Sec:3}. Section \ref{Sec:4} is devoted to the study of the algebra filtrations on the Hall algebra arising from degree functions imposed on the isomorphism classes of objects. In Section \ref{Sec:5}, we use the language of polyhedral cones to interpret the result in Section \ref{Sec:4} as a duality of polyhedral cones. Examples are discussed in Section \ref{Examples}. In Section \ref{extriangulated}, we discuss some outlooks on extriangulated categories which will be the topic of a sequel to this paper.

\subsection{Conventions}
Throughout the paper we keep the convention that the set of natural numbers $\mathbb{N}$ contains $0$. All categories are assumed to be essentially small.

\subsection{Acknowledgements}

M. G. is indebted to Hiroyuki Nakaoka, whose question led to the conception of this paper, and to Henning Krause for bringing his attention to the article \cite{Enomoto1}. The work that led to this paper goes back to the visit of M. G. to the University of Bielefeld in May-June 2018, and he is very grateful to Henning Krause for the invitation and to the entire BIREP group for valuable discussions and hospitality.

We are thankful to Ghislain Fourier and Markus Reineke for useful discussions and for the organisation of the workshop ``Geometry and representation theory at the interface of Lie algebras and quivers'' that took place in September 2018 at the Ruhr-Universit\"at Bochum and led to the collaboration. We are grateful to Eugene Gorsky, Bernhard Keller, Steffen K\"onig and Yann Palu for comments on preliminary versions of this paper.

\section{Exact structures on an additive category}\label{Sec:2}

\subsection{Exact categories} \label{Reminder_exact}

In an additive category $\mathcal{A}$ a pair of morphisms 
$$\xymatrix{{A\quad} \ar@{>->}[r]^i &*++{B} \ar@{->>}[r]^p &*++{C,}}$$ 
is said to be {\it exact}, or a {\it kernel-cokernel pair} if $i$ is a kernel of $p$ and $p$ is a cokernel of $i.$ An {\it exact category} in the sense of Quillen \cite{Q} is $\mathcal{A}$ endowed with a class $\mathcal{E}$ of exact pairs, closed under isomorphism and satisfying the following axioms: 

\begin{itemize}
\item[{[E0]}] For each object $E \in \mathcal{A},$ the identity morphism $1_E$ is an inflation.
\item[{[E0$^{op}$]}] For each object $E \in \mathcal{A},$ the identity morphism $1_E$ is a deflation.
\item[{[E1]}] The class of inflations is closed under composition.
\item[{[E1$^{op}$]}] The class of deflations is closed under composition.
\item[{[E2]}] The push-out of an inflation along an arbitrary morphism exists and yields
an inflation.
\item[{[E2$^{op}$]}] The pull-back of a deflation along an arbitrary morphism exists and yields
a deflation.
\end{itemize}

Here an {\it inflation} is a morphism $i$ for which there exists $p$ such that $(i,p)$ belongs to $\mathcal{E}.$ It is also called an {\it admissible monic}, or an {\it admissible monomorphism}. {\it Deflations} (or {\it admissible epics}, or {\it admissible epimorphisms}) are defined dually. We depict admissible monics
by $\rightarrowtail$ and admissible epics by $\twoheadrightarrow$ in diagrams. Exact pairs belonging to $\mathcal{E}$ are called {\it conflations}, or {\it admissible exact sequences}. An inflation being simultaneously a deflation is an isomorphism. By abuse of notation, we will denote an exact category $(\cA, \cE)$ simply by $\cE$ and call $\cA$ its {\it underlying additive category}.

\begin{remark}
This set of axioms is not minimal, cf. \cite{Kel1}, \cite{Buh}. We follow the terminology of \cite{GabrielRoiter, Kel1}. There are also slightly different notions of exact categories, but in case when $\mathcal{A}$ is {\it weakly idempotent complete}, all of them coincide, cf. \cite{Buh}. We will always assume that $\mathcal{A}$ is {\it idempotent complete}, that is an even stronger condition, cf. below. 
\end{remark}

Any abelian category has a canonical structure of an exact category. In this case, the class of conflations coincides with
the class of all exact pairs.
If $\mathcal{E}$ and $\mathcal{E'}$ are exact categories, an exact functor $\mathcal{E} \to \mathcal{E'}$ is an additive functor taking conflations
of $\mathcal{E}$ to conflations of $\mathcal{E'}.$
A fully exact subcategory of an exact category $\mathcal{E}$ is a full additive subcategory $\mathcal{E'} \subset \mathcal{E}$ which
is closed under extensions, i.e. if it contains the end terms of a conflation of $\mathcal{E},$ it also contains
the middle term. Then $\mathcal{E'}$ endowed with the conflations of $\mathcal{E}$ having their terms in $\mathcal{E'}$ is an exact
category, and the inclusion $\mathcal{E'} \hookrightarrow \mathcal{E}$ is a fully faithful exact functor. 

Note that one additive category may be endowed with a lot of different exact structures. One can consider any additive category $\mathcal{A}$ as an exact category with a split exact structure $\ce^{\add}$: namely, one can take the isomorphism closure in the class of kernel-cokernel pairs of the class
$$\left\{A \rightarrowtail A \oplus B \twoheadrightarrow B | A, B \in \mathcal{A}\right\}$$
as the set of conflations. 
If $\mathcal{A}$ has a structure of an abelian category that is not semi-simple, its split exact structure differs from the abelian one. All possible exact structures are between these two: the class of conflations is necessarily contained in the class of all exact pairs and contains all split exacts pairs. In other words, the split exact structure is the minimal one and the abelian exact structure is the maximal one. More discussion and the precise statement can be found in \cite{DReitenSS}, for the case with no abelian structure see \cite{Rump11} and \cite{Crivei12} and references therein. As shown by Rump \cite{Rump11}, any additive category has a unique maximal exact structure. We will always denote it by $\ce^{\mathrm{max}}$, the underlying additive category will be clear from the context. Br\"ustle, Hassoun, Langford and Roy \cite{BHLR} studied the poset of exact structures on an additive category. 

\begin{theorem} \cite[Theorem 5.3, Corollary 5.4]{BHLR}
The poset of exact structures on an additive category (with respect to the containment order on the corresponding classes of conflations) is a bounded complete lattice. Its maximal element is $\ce^{\max}$ and its minimal element is $\ce^{\add}.$
\end{theorem}

We say that $\mathcal{E}$ is {\it idempotent complete}, or that {\it idempotents split} in $\mathcal{E},$ if both $\mbox{Ker} (e)$ and $\mbox{Ker} (\mathrm{id} - e)$ exist in $\mathcal{E}$ for each idempotent $e: \mathcal{E} \to \mathcal{E}.$ Any abelian category is idempotent complete. Note that the idempotent completeness is a property of an additive category rather than of its exact structure. In particular, if we take an idempotent complete exact category and endow it with a bigger exact structure (by enlarging its class of conflations), the new exact category still will be idempotent complete. Also, if an additive category can be endowed with an abelian structure, it is idempotent complete.

We will assume that $\mathcal{E}$ is idempotent complete, linear over some field $k$ and essentially small. Suppose that moreover for all objects $A, B \in \mathcal{E},$ we have $\dim(\Hom(A, B)) < \infty.$ Then it is well-known that $\mathcal{E}$ is a Krull-Schmidt category. That is, each object decomposes into a finite direct sum of indecomposables in a unique way, up to permutation, and each of the indecomposables has a local endomorphism ring.

To each essentially small exact category $\cE$, one associates its Grothendieck group $K_0(\cE)$ defined in the same way as in the case of abelian categories: it is the free abelian group on the set of isomorphism classes $\Iso(\cE)$ modulo the relations $[B] = [A] + [C]$ for all conflations $A \rightarrowtail B \twoheadrightarrow C.$  In particular, to each additive category $\cA$, one associates the Grothendieck group of its split exact structure. It is called the {\it additive}, or {\it split} Grothendieck group of $\cA.$ We denote it by $K_0^{\mathrm{add}}(\cA),$ it is also often denoted $K_0^{split}(\cA)$ or $K_0(\cA, \oplus).$  As each exact category $\cE$ is additive, we have an additive Grothendieck group of its underlying additive category. By abuse of notation, we denote it $K_0^{\mathrm{add}}(\cE)$. It is known that for a Krull-Schmidt category, its additive Grothendieck group is freely generated by the classes of indecomposable objects. When an additive category is endowed with two different exact structures $\cE$ and $\cE',$ where $\cE$ is larger than $\cE'$ (i.e. the class of conflations in $\cE$ contains the class of conflations in $\cE'$), it follows that the Grothendieck group $K_0(\cE)$ is a quotient of the Grothendieck group $K_0(\cE').$ In particular, for each essentially small exact category $\cE,$ its Grothendieck group $K_0(\cE)$ is a quotient of $K_0^{\mathrm{add}}(\cE).$ If $\cE$ can be endowed with an abelian structure $\cC,$ then $K_0(\cC)$ is a quotient of $K_0(\cE).$ 

The notion of extensions and the Yoneda theory generalize naturally from the setting of abelian to that of exact categories. In particular, for a pair of objects $A, C$ in an essentially small exact category $\cE,$ its space of first extensions $\Ext_\cE^1(C,A)$ can be identified with the quotient of the set of all conflations 
$$A \hookrightarrow B \twoheadrightarrow C$$
by the same equivalence relation as in the abelian case. 

An object $P$ is said to be projective in $\cE$ if $\Ext_\cE^1(P, B) = 0$ for any $B \in \cA.$ This is not the usual definition, but an equivalent one. We denote by $\cp(\ce)$ the full subcategory of the projective objects in the category $\cE.$ 

\subsection{Finitely presented functors and defects}

Enomoto \cite{Enomoto1} classified exact structures on an idempotent complete additive category in terms of various functor categories. In this and in the next subsections, we recall some of results from this work and from the work of  Br\"ustle-Hassoun-Langford-Roy \cite{BHLR} and give certain improvements to them. Most of our proofs are essentially obtained by combining some arguments from \cite{Enomoto1} and \cite{Enomoto2}.

For an additive category $\cA$, a right $\cA-$module is a contravariant additive functor $\cA^{\op} \to \Ab$ to the category of abelian groups. The category of right $\cA-$modules is denoted by $\Mod \cA.$ The category of left modules is the category of covariant additive functors to abelian groups, it can be seen as $\Mod \cA^{\op}.$ Representable functors $\Hom(-, X) \in \Mod \cA$ and $\Hom(X,-) \in \Mod \cA^{\op}$ are projectives in these module categories. Moreover, both $\Mod \cA$ and $\Mod \cA^{\op}$ are abelian and have enough projectives, and these are nothing but direct summands of direct sums of representable functors. An $\cA$-module $M$ is finitely generated if admits an epimorphism $\Hom(-, X) \twoheadrightarrow M$ from a representable functor. It is moreover finitely presented if it admits an exact sequence $\Hom(-, X) \to \Hom(-, Y) \to M \to 0,$ for some $X, Y \in \cA,$ i.e. if it is a cokernel of a morphism of representable functors. It is well-known that $\cA$ is idempotent complete if and only if the essential image of the Yoneda embedding  $\cA\to\Mod\cA$ given by $X \mapsto \Hom(-, X)$ consists of all finitely generated projective $\cA-$modules. We denote by $\fp(\ca^{\op}, \Ab)$ the full subcategory of finitely presented modules in $\Mod \cA.$ We define finitely presented left modules dually, they form the category $\fp(\ca, \Ab).$

Let $\ce$ be an exact structure on $\ca.$ For any conflation $\delta: A \overset{f}\hookrightarrow B \overset{g}\twoheadrightarrow C$ in $\ce$, one has exact sequences of left, resp. right $\cA-$modules

$$0 \to \Hom(C, -) \overset{\Hom(g, -)}\longrightarrow \Hom(B, -) \overset{\Hom(f, -)}\longrightarrow \Hom(A,  -)$$

and 

$$0 \to \Hom(-, A) \overset{\Hom(-, f)}\longrightarrow \Hom(-, B) \overset{\Hom(-, g)}\longrightarrow \Hom(-, C).$$

Cokernels of the rightmost maps are the \emph{covariant defect}  $\delta_{\sharp}$, resp. the \emph{contravariant defect}  $\delta^{\sharp}$ of $\delta$, first considered by Auslander \cite{Aus1}, see also \cite[IV.4]{ARS}:

\begin{align} \label{proj_res_co}
0 \to \Hom(C, -) \overset{\Hom(g, -)}\longrightarrow \Hom(B, -) \overset{\Hom(f, -)}\longrightarrow \Hom(A, -) \to \delta_{\sharp}  \to 0; 
\end{align}

\begin{align} \label{proj_res_contra}
0 \to \Hom(-, A) \overset{\Hom(-, f)}\longrightarrow \Hom(-, B) \overset{\Hom(-, g)}\longrightarrow \Hom(-, C) \to \delta^{\sharp}  \to 0.
\end{align}

Since representable functors are projectives, sequences (\ref{proj_res_co}) and (\ref{proj_res_contra}) are projective resolutions of $\delta_{\sharp}$, resp. of $\delta^{\sharp}$. 

The defects have another interpretation: they are \emph{(weakly) effaceable functors with respect to inflations}, resp. \emph{deflations} \cite{KaledinLowen}, often called simply \emph{effaceable} \cite[Appendix A]{Kel1}, see also \cite{Fiorot, Enomoto1, Enomoto2}. A finitely presented right $\ca-$module $M$ is called \emph{$\cE-$effaceable} if 
 each $W \in \cE$ and $w \in M(W)$, there exists a deflation $g : E \twoheadrightarrow W$ in $\cE$ such that $(Mg)(w) = 0.$ Effaceable left functors are defined dually. 
 
 \begin{proposition} \cite[Proposition A.1]{Enomoto2}
 Let $\ca$ be an idempotent complete additive category.
 A finitely presented right $\ca-$module $M$ is $\ce-$effaceable if and only if there is a conflation $\delta$ in $\ce,$ such that $M = \delta^{\sharp}$. Dually, a finitely presented left $\ca-$module $N$ is $\ce-$effaceable if and only if there is a conflation $\delta$ in $\ce,$ such that $N = \delta_{\sharp}$. 
 \end{proposition}

Fiorot \cite{Fiorot} and Enomoto \cite{Enomoto1, Enomoto2} denote the category of right finitely presented $\ce-$effaceable modules by $\eff \cE.$ We denote the category of left $\ce-$effaceable modules by $\eff \cE^{\op}.$ In the setting of extriangulated categories with weak kernels, Ogawa \cite{Ogawa} denotes the category of contravariant defects by $\defff \ce$ and then proves that $\defff \ce = \eff \ce.$ 

Right $\ce-$effaceable functors form an (abelian) localizing Serre subcategory of $\Mod \cA,$ see \cite[Appendix A]{Kel1}. The dual statement naturally holds for left $\ce-$effaceable functors. The category $\fp(\ca^{\op}, \Ab)$ is abelian if and only if $\ca$ has weak kernels. Since deflations always have kernels, the category $\eff \ce$ behaves better than $\fp(\ca^{\op}, \Ab)$ in this aspect.

\begin{proposition} \label{abelian_Serre}
 Let $\ca$ be an idempotent complete additive category. Let $\ce$ be an exact structure on $\ca.$ 
 \begin{itemize}
\item [(i)]
The category $\eff \ce$ is closed under kernels, cokernels and extensions in $\Mod \cA$. In particular, it is abelian.
\item[(ii)]
Suppose that there exists an exact sequence 
$0 \to M_1 \to M \to M_2 \to 0$ in $\Mod \ca$ and that
$M_1$ is finitely generated. Then $M$ is in $\eff \ce$ if and only if both $M_1$ and $M_2$ are in $\eff \cE.$
 \end{itemize}
\end{proposition}

\begin{proof}
For the detailed proof of part (i), see \cite[Theorem A.2]{Enomoto2}. Part (ii) is proved in \cite[Proposition 2.10]{Enomoto1}.
\end{proof}

For each finitely presented functor $F \in \fp(\ca^{\op}, \Ab),$ Auslander \cite[II.3]{Aus3} defined the covariant functor $V(F): \ca \to \Ab$ by the rule $V(F)(X) = \Ext^2(F, \Hom(-, X)).$ Dually, for each $G \in \fp(\ca, \Ab),$ he defined the contravariant functor $W(G): \ca^{\op} \to \Ab$ by the rule $W(G)(X) = \Ext^2(G, \Hom(X, -)).$ From the form of projective resolutions (\ref{proj_res_co}) and (\ref{proj_res_contra}) it follows that for each conflation $\delta$, one has

\begin{align} \label{dualities}
V(\delta^{\sharp}) = \delta_{\sharp}, \qquad W(\delta_{\sharp}) = \delta^{\sharp}.
\end{align}

Moreover, these dualities respect morphisms of effaceable functors, and so $V: (\eff \cE)^{\op} \to \eff \ce^{\op}$ and $W: (\eff \ce^{\op})^{\op} \to \eff \ce$ are mutually inverse equivalences, see e.g. \cite[Theorem 3.4 in Chapter II]{Aus3}.

\begin{lemma} \label{eff_inclusion}
For two exact structures $\cE$ and $\cE'$, we have:
$$\cE' \subseteq \cE \Longleftrightarrow \eff \cE' \subseteq \eff \cE \Longleftrightarrow \eff \ce'^{\op} \subseteq \eff\cE^{\op}.$$
\end{lemma}

\begin{proof}
This follows from the description of effaceable functors as defects of conflations and from the duality (\ref{dualities}).
\end{proof}

Not all finitely presented $\cA-$modules can arise as defects of conflations in exact structures on $\cA.$ In the following two statements, Enomoto classifies all such modules and, further, all categories that appear as $\eff \ce$ for some $\ce.$

\begin{lemma} \cite[Lemma 2.3]{Enomoto1}
For a right $\ca-$module, the following are equivalent:
\begin{itemize}
\item[(i)]
There exists a kernel-cokernel pair $A \to B \overset{g}\to C$ in $\ca$ such that $M$ is isomorphic to
$\coker(\Hom(-, g)).$
\item[(ii)]
There exists an exact sequence 
$$0 \to \Hom(-,A) \to \Hom(-, B) \to \Hom(-,C) \to M \to 0$$
in $\Mod \ca$ and $\Ext^i (M, \Hom(-, X)) =
0$ for $i = 0, 1$ and all $X \in \ca$.
\end{itemize}
\end{lemma}

Enomoto \cite{Enomoto1} denotes by $\cc_2(\ca)$ the full subcategory of $\Mod \ca$ consisting of modules satisfying the above equivalent conditions. 

\begin{theorem} \cite[Theorem 2.7]{Enomoto1} \label{Enomoto_classif}
Let $\ca$ be an idempotent complete additive category. Then there exist mutually
inverse bijections between the following two classes.
\begin{itemize}
\item[(i)]
Exact structures $\ce$ on $\ca$.
\item[(ii)]
Subcategories $\cd$ of $\cc_2(\ca)$ satisfying the following conditions.
\begin{itemize}
    \item [(a)]
    Suppose that there exists an exact sequence $0 \to M_1 \to M \to M_2 \to 0$ of finitely generated right $\ca$-modules. Then $M$ is in $\cd$ if and only if both $M_1$ and $M_2$ are in $\cd.$
    \item [(b)]
    Suppose that there exists an exact sequence $0 \to M_1 \to M \to M_2 \to 0$ of finitely generated left $\ca$-modules. Then $M$ is in $\Ext^2(\cd, \Hom(-, \ce))$ if and only if both $M_1$ and $M_2$ are in $\Ext^2(\cd, \Hom(-, \ce))$.
\end{itemize}
\end{itemize}
Under the bijection, an exact structure $\ce$ maps to the category $\eff \cE.$ A subcategory $\cd$ of $\cc_2(\ca)$ maps to the structure $\ce(\cd)$ consisting of all kernel-cokernel pairs $A \overset{f}\to B \overset{g}\to C$ in $\ca,$ such that there exists an exact sequence in $\Mod \ca$
$$0 \to \Hom(-, A) \overset{\Hom(-, f)}\longrightarrow \Hom(-, B) \overset{\Hom(-, g)}\longrightarrow \Hom(-, C) \to M \to 0$$
with $M \in \cd.$
\end{theorem}

By combining these arguments, we can give a simpler classification.
 
\begin{theorem} \label{classif_general:body}
The lattice of exact structures on an idempotent complete additive category $\cA$ is isomorphic to the lattice of Serre subcategories of $\eff \ce^{\max}$ and to the lattice of Serre subcategories of $\eff (\ce^{\max})^{\op}$. 
\end{theorem}

\begin{proof}
Let $\ce$ be an exact structure on $\ca$. By Lemma \ref{eff_inclusion}, we have $\eff \cE \subseteq \eff \cE^{\max}$. By Proposition \ref{abelian_Serre}(i), both $\eff \cE$ and $\eff \cE^{\max}$ are abelian full extension-closed subcategories in $\Mod \cA.$ All functors in $\eff \cE^{\max}$ and, in particular, in $\eff \cE$ are finitely generated. Then by Proposition \ref{abelian_Serre}(ii) it immediately follows that $\eff \cE$ is a Serre subcategory of $\eff \ce^{\max}.$

Let $\cd$ be a Serre subcategory of $\ce^{\max}.$ By the $\Ext^2-$duality, its $\Ext^2-$dual is a Serre subcategory of $\eff {\cE^{\max}}^{\op}$. Since all functors in $\eff \ce^{\max}$ are in $\cC_2(\cA)$ and finitely generated, so are all functors in $\cd.$ The dual holds for the $\Ext^2-$dual of $\cd.$ Then by Theorem \ref{Enomoto_classif}, $\ce(\cd)$ is an exact structure on $\cA.$
\end{proof}

\begin{corollary} \label{classif_interval}
Let $\cE$ be an exact structure on an idempotent complete additive category $\cA.$ The interval $[\ce^{\add}, \ce]$ in the lattice of exact structures on $\cA$ is isomorphic to the lattice of Serre subcategories of $\eff \ce$ and to the lattice of Serre subcategories of $\eff \ce^{\op}.$
\end{corollary}

\begin{remark}
If the class of all kernel-cokernel pairs forms an exact structure, it is then the maximal one. This happens, in particular, when $\ca$ is abelian or, more generally, quasi-abelian (\cite{Schneiders}, see also \cite[Section 4]{Buh}). In this case, $\eff \ce^{\max} = \cc_2(\ca),$ and we recover \cite[Lemma 2.12]{Enomoto1}: the lattice of exact structures on $\ca$ is isomorphic to the lattice of Serre subcategories of $\cc_2(\ca)$.
\end{remark}

If $\ce$ has enough projectives, the ideal quotient of $\ca$ by all the morphisms that factor through projective objects is called the \emph{projectively stable} category of $\ce$ and denoted by $\underline{\ce}.$ 
Dually one defines the \emph{injectively stable}, or the \emph{costable} category of $\ce$ denoted by $\overline{\ce}.$ 

\begin{lemma} \cite[Lemma 2.13]{Enomoto1}
If $\ce$ is an exact structure with enough projectives, then $\eff \ce \overset\sim\to \fp({\underline{\ce}}^{\op}, \Ab).$
\end{lemma}

\begin{corollary} (cf. \cite[Corollary 2.14]{Enomoto1})
If $\ce$ is an exact structure with enough projectives, then the interval $[\ce^{\add}, \ce]$ is isomorphic to the lattice of Serre subcategories of the category $\fp({\underline{\ce}}^{\op}, \Ab).$
\end{corollary}

In particular, this recovers Buan's classification of exact structures on the category  $\mod \Lambda$ for an Artin algebra $\Lambda$ \cite[Proposition 3.3.2]{Buan}. Note that Buan was working in terms of closed additive sub-bifunctors of $\Ext^1_{\Lambda}(-, -)$, cf. Section 1.3.

Dually, if $\ce$ is an exact structure with enough injectives, then the interval $[\ce^{\add}, \ce]$ is isomorphic to the lattice of Serre subcategories of the category $\fp({\overline{\ce}}^{\op}, \Ab).$

\subsection{AR conflations and admissible exact categories} \label{AR+admissible}

Recall the basic notions of the Auslander-Reiten (AR) theory in the setting of exact categories. We follow \cite[Section 9]{GabrielRoiter}, see also \cite{ARIII, ARS}. Throughout this section, we assume that $\ca$ is a Krull-Schmidt (additive) category. Let $\cj$
be its Jacobson radical. For an object $X \in \cA$, we set $S_X := \Hom(-, X)/\cj(-, X).$ It is well-known that the map $X \to S_X$ is a bijection between $\Ind(\ca)$ and the set of isomorphism classes of simple right $\ca-$modules (see e.g. \cite{Aus2}). The dual statement holds for the left modules $S^X := \Hom(X, -)/\cj(X, -).$

Let $X, Z$ be indecomposable objects in $\ca.$ Let $\delta: X \overset{f}\rightarrowtail Y \overset{g}\twoheadrightarrow Z$ 
be a conflation in $\ce$. By abuse of notation, we also denote by $\delta$ its class in $\Ext^1_{\ce}(Z, X).$

\begin{theorem} \label{AR:def/thm}
\cite[Theorem 9.3]{GabrielRoiter}
The following are equivalent.
\begin{itemize}
    \item 
    The sequence
    $$0 \to \Hom(-, X) \overset{\Hom(-, f)}\longrightarrow \Hom(-, Y) \overset{\Hom(-, g)}\longrightarrow \Hom(-, Z) \to S_Z \to 0
$$
is the minimal projective resolution of $S_Z$;
    \item
    $g$ is not a retraction and every morphism in $\cj(-, Z)$ factors through $g$;
    \item
    The sequence
    $$0 \to \Hom(Z, -) \overset{\Hom(g, -)}\longrightarrow \Hom(Y, -) \overset{\Hom(f, -)}\longrightarrow \Hom(X, -) \to S^X \to 0
$$
is the minimal projective resolution of $S^X$;
    \item
    $f$ is not a section and every morphism in $\cj(X, -)$ factors through $f$.
\end{itemize}
The conflation $\delta,$ if exists, is uniquely determined (up to isomorphism) by the object $X$. It is also uniquely determined by the object $Z$.
\end{theorem}

Conflations $\delta$ satisfying the equivalent conditions of Theorem \ref{AR:def/thm} are called \emph{almost split}, or \emph{Auslander-Reiten} (AR) conflations.
When $\Ext^1(Z, X)$ contains an AR conflation, one uses the notation $X = \tau Z$ and $Z = \tau^{-1} X.$

We say that $\ce$ has AR conflations if for every indecomposable non-projective object $Z$ there exists an AR conflation ending at $Z$, and for every indecomposable non-injective object $X$ there exists an AR conflation starting at $X$.

\begin{proposition} \label{Prop:Enomoto2:2.3}
\cite[Proposition 2.3]{Enomoto2}
Let $\ca$ be a Krull-Schmidt additive category, $Z$ be an indecomposable object and $\ce$ an exact structure on $\ca$. Then the following are equivalent.
\begin{itemize}
\item[(i)] There is an $\AR$ conflation in $\ce$ ending at $Z$;
\item[(ii)] $S_Z$ is in $\eff \ce$;
\item[(iii)] $S_Z$ is finitely presented and $Z \notin \cp(\ce).$
\end{itemize}
\end{proposition}

Simple objects in $\eff \ce$ are precisely the simple right $\ca-$modules that belong to $\eff \ce$. An object in $\eff \ce$ has finite length in $\eff \ce$ if and only if it has finite length in $\Mod \ca$, and the composition series are the same in this case. (\cite[Proposition A.3]{Enomoto2}). The dual of the statement of Proposition \ref{Prop:Enomoto2:2.3} holds for the left modules $S^Z$ in $\eff \ce^{\op}.$

\begin{corollary} \label{AR:simple}
\begin{itemize}
    \item[(i)] The map $Z \to S_Z$ is a bijection between the set $\Ind(\ca) \backslash \ind(\cp(\ce))$ and the set of isomorphism classes of simple objects in $\eff \ce;$
    \item[(ii)] The simple objects $S_Z$ in $\eff \ce$ are precisely the contravariant defects of $\AR-$conflations in $\ce$;
    \item[(iii)] The map $Z \to S^Z$ is a bijection between the set $\Ind(\ca) \backslash \ind(\cp(\ce))$ and the set of isomorphism classes of simple objects in $\eff \ce^{\op};$
    \item[(iv)] The simple objects $S_Z$ in $\eff \ce^{\op}$ are precisely the covariant defects of $\AR-$conflations in $\ce$.
\end{itemize}
\end{corollary}

Let $\AR_+(\ce)$ and $\AR(\ce)$ be the sub-semigroup, resp. subgroup of $K_0^{\mathrm{add}}(\ce)$ generated by
$$\left\{[X] - [Y] + [Z]\ | \text{ there exists an AR conflation } X \rightarrowtail Y \twoheadrightarrow Z  \mbox{ in } \ce \right\}.$$
Let $\Ex_+(\ce)$ and $\Ex(\ce)$ be the sub-semigroup, resp. subgroup of $K_0^{\mathrm{add}}(\ce)$ generated by 
$$\left\{[X] - [Y] + [Z] \  | \text{ there exists a conflation } X \rightarrowtail Y \twoheadrightarrow Z \mbox{ in } \ce \right\}.$$

Enomoto \cite{Enomoto1} calls an exact structure $\cE$ \emph{admissible} if every object in $\eff \ce$ has finite length.

Let $F$ be a right or a left module over a Krull-Schmidt category $\ca$. Its \emph{support} is defined in terms of indecomposable objects in $\ca$:
$$\supp(F) := \left\{X \in \Ind(\ca) \ |  F(X) \neq 0 \right\}.$$

\begin{lemma} \label{length-support}
If $\cA$ is a $\Hom-$finite Krull-Schmidt category, then a finitely generated $\cA-$module has finite length if and only it has finite support.
\end{lemma}

\begin{proof}
See, e.g. \cite[Lemma 3.4]{KrauseVossieck} or \cite[Theorem 2.12]{Aus2}.
\end{proof}

We will need the following statement in Section 4.

\begin{proposition} \cite[Proposition 2.2]{Aus2}\label{supp_simple}
Let $X$ be an indecomposable object. Then the supports of the right simple module $S_X$ and the left simple module $S^X$ are both equal simply to $\left\{X\right\}.$ Moreover, we have
$$S_X(Y) = S^X(Y) = \delta_X^Y, \quad \forall Y\in \Ind(\ca).$$
\end{proposition}

We say that an additive category $\ca$ is \emph{additively finite} if it has only finitely many indecomposable objects up to isomorphism. As before, we assume that $\ca$ is $\Hom-$finite and linear over a field $k$. Then Enomoto proved in \cite{Enomoto1} that all exact structures on $\cA$ are admissible and can be classified via Auslander-Reiten (AR) theory. For the convenience of the reader, we give precise references to those results of Enomoto for additively finite categories that are important for us.

\begin{theorem} [Enomoto] \label{enomoto} 
Let $\cA$ be a $\Hom-$finite idempotent complete $k-$linear additively finite category. Let $\ce$ be an exact structure on $\cA.$ Then the following holds. 
\begin{itemize}
\item[(i)] \cite[Corollary 3.15]{Enomoto1} $\ce$ has enough projectives and enough injectives.
\item[(ii)] \cite[Corollary A.4]{Enomoto1} $\ce$ has AR conflations.
\item[(iii)] \cite[Proposition 3.8(2)]{Enomoto1}  $\ce$ is $\Ext^1-$finite.
\item[(iv)] \cite[Proposition 3.8(3)]{Enomoto1} $\ce$ is admissible.
\item[(v)] \cite[Corollary 3.10]{Enomoto1} There is a one-to-one correspondence between the exact structures on $\cA$ and the subsets of the (finite) set $\ind(\ca) \backslash \ind(\cp(\ce^{\max})).$ For an exact structure $\ce,$ the corresponding subset coincides with the set $\ind(\cp(\ce))  \backslash \ind(\cp(\ce^{\max})).$ In particular, $\ind(\cp(\ce^{\mathrm{add}})) = \ind(\ca).$
\item[(vi)]\cite[Corollary 3.18]{Enomoto1}  $\Ex_+(\ce) = \AR_+(\ce).$ 
\end{itemize}
\end{theorem}

The last result was stated as $\Ex(\ce) = \AR(\ce),$  but the proof used only the fact that the defect of each conflation in $\ce$ has finite length, and this means that it actually applies to semigroups. While this might look like a minor distinction, it will be crucial for the proofs of our main results.

In paper \cite{Enomoto2}, Enomoto further investigated the relation between parts (iv) of (vi) of Theorem \ref{enomoto} without the finiteness assumption on the underlying additive category, i.e. between the admissibility of an exact category $\ce$ and the property $\Ex_+(\ce) = \AR_+(\ce)$. Under a certain condition (of \emph{conservation of finiteness} (CF) \cite[Definition 3.5]{Enomoto2}), he proved that $\cE$ is admissible if and only if $\ce$ has $\AR$ conflations and $\Ex_+(\ce) = \AR_+(\ce)$. We are mainly concerned with $\Hom-$finite additive categories, and the condition (CF) is satisfied for all exact structures on them (essentially by Lemma \ref{length-support}).

\begin{theorem}[\cite{Enomoto2}]\label{Thm:Eno2} If (the underlying additive category of) an exact category $\cE$ is $\Hom-$finite, then $\cE$ is admissible if and only if $\Ex_+(\cE) = \AR_+(\cE).$ 
\end{theorem}

We should again emphasize that Enomoto \cite{Enomoto2} proved that $\Ex(\cE) = \AR(\cE)$, but the proof remains the same for sub-semigroups. 

\begin{remark}
It follows from Theorem \ref{Thm:Eno2} and its proof that for exact structures on $\Hom-$finite categories, conditions $\Ex(\cE) = \AR(\cE)$ and $\Ex_+(\cE) = \AR_+(\cE)$ are equivalent. Indeed, if $\Ex(\cE) = \AR(\cE)$, then, by Theorem \ref{Thm:Eno2} (in the original statement of \cite{Enomoto2}), the structure $\ce$ is admissible. Then by the proof of Theorem \ref{Thm:Eno2}, $\Ex_+(\cE) = \AR_+(\cE)$. The converse implication is straightforward. The same is true for all exact structures satisfying the condition (CF).
\end{remark}

Zhu and Zhuang \cite{ZZ} call a Krull-Schmidt category $\ca$ \emph{locally finite} if every representable left and every representable left indecomposable $\ca$-module has finite support:
$$|\supp(\Hom(X, -))| < \infty, \qquad |\supp(\Hom(-, X))| < \infty, \qquad \forall X \in \Ind(\ca).$$
Since the category is Krull-Shmidt, this happens if and only if  every representable left and every representable left $\ca$-module has finite support:
$$|\supp(\Hom(X, -))| < \infty, \qquad |\supp(\Hom(-, X))| < \infty, \qquad \forall X \in \Iso(\ca).$$
There are slightly different notions of local finiteness of additive categories, but they coincide with this definition for $\Hom-$finite Krull-Schmidt categories. Zhu and Zhuang prove the following theorem for extriangulated structures on additive categories (see Section 7). We give an alternative proof for exact structures. Again, the statement in fact concerns semigroups.

\begin{theorem} [\cite{ZZ}] \label{ZZ} Let $\cA$ be a $\Hom-$finite, locally finite, Krull-Schmidt category and let $\ce$ be an exact structure on $\cA$. Then $\cE$ has $\AR-$conflations and $\Ex_+(\cE) = \AR_+(\cE)$
\end{theorem}

Zhu and Zhuang assume the structure to be $\Ext^1-$finite, but this condition is not necessary at least in the exact case.

\begin{proof}
 Local finiteness implies that every representable functor has finite support. Since $\ca$ is $\Hom-$finite, by Lemma \ref{length-support} it has finite length. Then for each exact structure $\cE$ each effaceable functor also has finite length in $\Mod \ca$ as a quotient of an object of finite length. By \cite[Proposition A.3]{Enomoto2}, it has finite length in $\eff \ce$. By definition, $\cE$ is admissible. By Theorem \ref{Thm:Eno2},  $\Ex_+(\cE) = \AR_+(\cE).$
\end{proof}

\begin{example} \label{Geigle}
Let $\Lambda$ be a hereditary Artin algebra of infinite representation type. Its Auslander-Reiten quiver decomposes into the preprojective component, a collection of regular components and the preinjective component. Let $\mathscr{P}$ be the full subcategory of $\mod \Lambda$ defined by the preprojective component.  Clearly, the maximal exact structure $\cE^{\max}$ on $\mathscr{P}$ is the one induced from $\mod \Lambda.$ It was implicitly proved already by Geigle in \cite{Ge} and further discussed in \cite[Remark 2]{MMP} and \cite[Example 4.5(3)]{ZZ} that $\Ex(\cE^{\max}) = \AR(\cE^{\max}).$ The property $\Ex(\cE) = \AR(\cE)$ holds then for any exact structure $\cE$ on $\mathscr{P}.$ Note that $\mathscr{P}$ is a locally finite additive category, so this property also follows from the general result of Zhu and Zhuang \cite{ZZ}. Since $\Lambda$ is of infinite representation type, so is $\mathscr{P}.$ To summarize, the exact structure $\cE^{\max}$ on $\mathscr{P}$ provides the first example of an additive category which is not additively finite, while its maximal exact structure is admissible (and satisfies the property $\Ex_+ = \AR_+$).
\end{example}

\begin{example}
The split exact structure $\ce^{\add}$ on any additive category $\cA$ is always admissible: its category of defects has only one object, the zero functor. Every object in $\ca$ is projective-injective with respect to $\ce^{\add}$. Thus, each additive category has at least one admissible exact structure. This emphasizes the fact that the admissibility of an exact structure $\cE$ on an additive category $\cA$ depends on the $\cE$ itself and does not impose conditions on $\cA$.
\end{example}

\begin{example}
Let $\Lambda$ be a hereditary Artin algebra of infinite representation type, as in Example \ref{Geigle}. The maximal exact structure on $\mod \Lambda$ is not admissible. However, it has AR conflations. Their defects form the sets of simple covariant, resp. simple contravariant functors from $\mod \Lambda$ to the category $\mathbf{Ab}.$ Any Serre subcategory of $\mod(\mod \Lambda)$ generated by a subset of simples (i.e. of defects of some of AR conflations) via finitely many extensions is automatically a finite length category, so the corresponding exact structure on $\mod \Lambda$ is admissible. These can be seen as natural examples of non-split admissible exact structures on additive categories that are not locally finite. 
\end{example}

The last example can be generalized to the case of arbitrary Krull-Schmidt additive categories.

\begin{theorem} \label{max_adm:body}
Every Krull-Schmidt additive category $\cA$ has a unique maximal admissible exact structure. 
\end{theorem}

\begin{proof}
Take the maximal exact structure $(\ca, \ce^{\max})$. Its category of defects $\eff \ce^{\max}$ is abelian (Proposition \ref{abelian_Serre}(i)). Then its full subcategory of finite length objects is a Serre subcategory. 
By Theorem \ref{classif_general:body} it defines a (unique) exact structure $\cE$ on $\cA$. This exact category is admissible by construction, and again by construction any other admissible exact structure on $\cA$ is smaller.
\end{proof}

\begin{theorem} \label{classif_adm:body} 
Let $\cE$ be an admissible exact category on a Krull-Schmidt category $\cA.$ Then the following lattices are isomorphic:
\begin{itemize}
    \item The interval $[\ce^{\add}, \ce]$ in the lattice of exact structures;
    \item The lattice of subsets of the set of AR conflations in $\cE;$
    \item The lattice of subsets of the set of all the indecomposables in $\cA$ that are not projective in $\cE;$
    \item The lattice of subsets of the set of all the indecomposables in $\cA$ that are not injective in $\cE.$
\end{itemize}
In particular, the interval $[\ce^{\add}, \ce]$ is a Boolean lattice.
\end{theorem}

\begin{proof}
Since $\ce$ is admissible, every object in $\eff \ce$ has finite length. By Corollary \ref{classif_interval}, the interval $[\ce^{\add}, \ce]$ is isomorphic to the lattice of Serre subcategories of $\eff \ce.$
Every Serre subcategory of a category where each object has finite length is a category of all objects with (automatically finite) composition series with terms in some subset of the set of simples, this defines a bijection between Serre subcategories of the category $\eff \ce$ and subsets of its set of simple objects. This bijection preserves the contaimnent order, it is thus a lattice isomorphism. By Corollary \ref{AR:simple}, the lattice of subsets of the simple objects in $\eff \ce$ is isomorphic to the lattice of subsets of the set of AR conflations in $\cE$
 and to the lattice of subsets of the set $\Ind(\ca) \backslash \Ind(\cp(\ce)).$ 
\end{proof}

For an exact structure $\ce$ and a full subcategory $\cb \subset \ca,$ one can associate a pair of exact structures $\ce_{\cb}, \ce^{\cb} \in [\ce^{\add}, \ce]$ as follows: a conflation $\delta: X \rightarrowtail Y \twoheadrightarrow Z$ in $\ce$ is also a conflation in $\ce_{\cb}$ if $\delta^{\sharp}(B) = 0, \ \forall B \in \cb.$ Dually, it is in  $\ce^{\cb}$ if $\delta_{\sharp}(B) = 0, \ \forall B \in \cb.$ One can reformulate these conditions in terms of $\Ext^1-$bifunctors:
\begin{align*}
\Ext^1_{\ce_{\cb}}(Z, X) = \left\{\delta \in \Ext^1_{\ce}(Z, X) \ | \ \Hom(\cb, Y) \to \Hom(\cb, Z) \to 0 \ \text{is exact} \right\}; \\
\Ext^1_{\ce^{\cb}}(Z, X) = \left\{\delta \in \Ext^1_{\ce}(Z, X) \ | \ \Hom(Y, \cb) \to \Hom(X, \cb) \to 0 \ \text{is exact} \right\}. \end{align*}

Note that all objects in $\cb$ become projective in $\ce_{\cb}$ and injective in $\ce^{\cb}.$

\begin{proposition} \label{relative_subcat}
For an exact structure $\ce$ on a Krull-Schmidt category $\ca$, the following are equivalent:
\begin{itemize}
\item[(i)] 
$\ce$ is admissible;
\item[(ii)]
For each $\ce' \in [\ce^{\add}, \ce]$, there exists a full subcategory $\cb \subset \ca,$ such that $\ce' = \ce^{\cb};$
\item[(iii)]
For each $\ce' \in [\ce^{\add}, \ce]$, there exists a full subcategory $\cb' \subset \ca,$ such that $\ce' = \ce_{\cb'}.$
\end{itemize}
\end{proposition}

\begin{proof}
We prove the equivalence of conditions (i) and (ii), the equivalence (i) $\Leftrightarrow$ (iii) is proved by dual arguments.

{\bf (i) $\Rightarrow$ (ii)}: Assume that $\ce$ is admissible. Let $\ce'$ be a smaller exact structure. By Theorem \ref{classif_adm:body}, it canonically defines a subset $S_1$ of the set $\Ind(\ca) \backslash \Ind(\cp(\ce))$. Let $\cb$ be the full subcategory of $\ca$ of all finite direct sums of elements of $\Ind(\ca) \backslash S_1.$ By direct verification, $\ce' = \ce^{\cb}.$

{\bf (ii) $\Rightarrow$ (i)}: We prove the implication by contradiction. Assume that $\ce$ is not admissible. The full subcategory of all objects of finite length in $\eff \ce$ is a Serre subcategory. By Corollary \ref{classif_interval}, it canonically defines an exact structure $\ce' < \ce$. Since $\eff \ce'$ has the same simple objects as $\eff \ce,$ the sets of projectives $\cp(\ce')$ and $\cp(\ce)$ coincide by Corollary \ref{AR:simple}.
Consider an arbitrary full subcategory $\cb \subset \ca$. If $\cb \subset \cp(\ce)$, then $\ce^{\cb} =\ce.$ If $\cb \not\subset \cp(\ce)$, then $\ce^{\cb}$ has strictly more projectives than $\ce.$ Therefore, there is no full subcategory $\cb \subset \ca$ such that $\ce = \ce^{\cb}.$
\end{proof}

\section{Hall algebras}\label{Sec:3}

Let $\mathcal{E}$ be an essentially small exact category, linear over a finite field $k.$ Assume that $\mathcal{E}$ has finite morphism and (first) extension spaces:
$$|\mbox{Hom}_\mathcal{E}(A, B)| < \infty, \quad  |\Ext^1_\mathcal{E} (A, B)| < \infty, \quad \forall A, B \in \mathcal{E}.$$
Given objects $A, B, C \in \mathcal{A},$ define $\Ext^1_\mathcal{E} (A, C)_B \subset \Ext^1_\mathcal{E} (A, C)$ as the subset parameterizing extensions whose middle term is isomorphic to B. We define the Hall algebra $\mathcal{H(E)}$ to be the $\mathbb{Q}-$vector space whose basis is formed by the isomorphism classes $[A]$ of objects $A$ of $\mathcal{E},$
with the multiplication defined by
$$[A] \diamond [C] = \sum\limits_{B \in \mbox{\footnotesize{Iso}}(\mathcal{E})} \frac{|\Ext^1_\mathcal{E} (A, C)_B|}{|\mbox{Hom}_\mathcal{E} (A, C)|} [B].$$

The following result was proved by Ringel \cite{R1} for $\mathcal{E}$ abelian, and later by Hubery \cite{Hu} for $\mathcal{E}$ exact. The definition of $\mathcal{H(E)}$ is also due to Ringel.

\begin{theorem}[\cite{R1, Hu}]
The algebra $\mathcal{H(E)}$ is associative and unital. The unit is given by $[0]$, where $0$ is the zero object of $\mathcal{E}$.
\end{theorem}

\begin{remark}
The choice of the structure constants
$\frac{|\Ext^1_{\mathcal{E}}(A,C)_{B}|}{|\Hom_{\mathcal{E}}(A,C)|} $ is the one that was used by Bridgeland \cite{Br}.
This choice is equivalent to that of the usual structure constants
$|\{ B' \subset C | B' \cong B, C/B' \cong A  \}|,$
called the {\it Hall numbers} and appearing in \cite{R1,Sch,Hu}. 
See \cite[\S2.3]{Br} for the details. More precisely, the usual Hall algebra carries a natural coalgebra structure. Our choice of the structure constants actually defines the dual algebra of the Hall coalgebra. It is known that this dual algebra and the usual Hall algebra are naturally isomorphic, see \cite{XX2} for a detailed discussion. 
\end{remark}

For twisted Hall algebras defined below, one has to tensor $\mathcal{H(E)}$  with $\mathbb{Q}(\nu),$ for $\nu = \sqrt{q},$ and consider it as an algebra over $\mathbb{Q}(\nu).$ By abuse of notation, we will use the same notation $\cH(\cE)$ for this new algebra, and we will not usually specify which of the two we consider. 

Assume that $\mathcal{E}$ is locally homologically finite and that all higher extension spaces are finite: for any $A, B \in \mathcal{E}$ and $p\geq 0$, $|\Ext^p_\mathcal{E} (A, B)| < \infty$; and there exists $p_0\geq 0$ such that for any $p>p_0$, $\Ext^p_\mathcal{E} (A, B) = 0$.

For objects $A, B \in \mathcal{E}$, we define the (multiplicative) Euler form 
$$\left\langle A, B \right\rangle_{\ce} := \prod_{i = 0}^{\infty} |\Ext^i_{\mathcal{E}}(A,B)|^{{(-1)}^i}.$$

From the existence of long exact sequences, it follows that this form descends to a bilinear form on the Grothendieck group $K_0(\mathcal{E})$ of $\mathcal{E}$,
denoted by the same symbol:
$$\left\langle \cdot, \cdot \right\rangle_{\cE}: K_0(\mathcal{E}) \times K_0(\mathcal{E}) \to \mathbb{Q}^{\times}.$$

One can also define the additive Euler form by taking the alternating sum of dimensions of extension spaces:
$$\left\langle A, B \right\rangle_{\cE, a} := \sum_{i=0}^\infty  (-1)^i\dim(\Ext^i_{\mathcal{E}}(A,B)).$$

If $\cE$ is an $\mathbb{F}_q-$linear category, we get $\left\langle A, B \right\rangle_{\cE} = q^{\left\langle A, B \right\rangle_{\ce, a}}.$

To avoid overlong notations, we will write simply $\left\langle A, B \right\rangle$ and $\left\langle A, B \right\rangle_{a}$ whenever the category $\cE$ is clear from the context.

Given any bilinear form on $K_0(\cE)$ with values in $\mathbb{Q}(\nu)^{\times},$ one can define a $\nu-$deformation of the Hall algebra with a twisted multiplication. The most important one is given by the square root of the multiplicative Euler form.
Namely, the {\it twisted Hall algebra} $\mathcal{H}_{tw}(\mathcal{E})$
is the same vector space over $\mathbb{Q}(\nu)$ as $\mathcal{H}(\mathcal{E}),$
with the twisted multiplication
$$[A] * [B] := \sqrt{\left\langle A, B \right\rangle} \cdot [A] \diamond [B] = \nu^{\left\langle A, B \right\rangle_{a}} \cdot [A] \diamond [B], \quad \forall  A, B \in \Iso(\mathcal{E}).$$

\section{Exact structures and degenerations of Hall algebras}\label{Sec:4}

Let $\cA$ be an additive category. As in the previous sections, we assume that $\cA$ is essentially small, $\Hom-$finite, idempotent complete and linear over some field $k.$ We know that $\cA$ can be endowed with many different exact structures. Assume that $\cE'$ and $\cE$ are two different $\Ext^1-$finite exact structures. Then their Hall algebras $\cH(\cE')$ and $\cH(\cE)$ are both well-defined. It is then natural to ask whether there is any sensible relation between these algebras. In this section, we find such a relation in the case when $\cE$ is larger than $\cE',$ i.e. it has more conflations. We denote it by $\cE' < \cE.$  Note that in this case, $\cE'$ has more projective objects than $\cE$ and, in particular, more indecomposable projectives: $\ind(\cp(\cE)) \subset \ind(\cp(\cE')),$ where $\ind(-)$ denotes the set of indecomposable objects in a category. We denote by $\ce \backslash \ce'$ the class of conflations in $\ce$ that are \emph{not} conflations in $\ce'.$  Proceeding from $\ce$ to $\ce'$ is called the \emph{reduction} of exact structures in \cite{BHLR}, see the reference for motivation.

\subsection{General result}

Generalizing a definition in \cite{FFR1} from the category of quiver representations (considered as an abelian category), we introduce the following notions.

\begin{definition}
A function $w: \Iso(\cA) \to \mathbb{N}$ is called
\begin{itemize}
\item[(i)] {\it normalized} if $w(M) = 0$ only for $M = 0$,
\item[(ii)] {\it an $\cE-$quasi-valuation} if $w(X) \leq w(M \oplus N)$ whenever there exists a conflation $$N \rightarrowtail X \twoheadrightarrow M$$ in $\cE$;
\item[(iii)] {\it $\ce-$additive} if  $w(X) = w(M) + w(N)$ whenever there exists a conflation $$N \rightarrowtail X \twoheadrightarrow M$$ in $\cE$;
\item[(iv)] {\it an $\ce-$valuation} if it is an $\cE-$quasi-valuation and $\cE^{\mathrm{add}}-$additive, i.e. we have $w(M \oplus N) = w(M) + w(N)$ for all $M$ and $N$;
\item[(v)] {\it $\ce'$-characteristic in $\ce$} if it is $\ce'-$additive and $w(X) < w(M)+w(N)$ whenever
there exists a conflation 
$$N \rightarrowtail X \twoheadrightarrow M$$ 
in $\ce \backslash \ce'$.
\end{itemize}
\end{definition}

Note that, for $\ce' < \ce,$ any  $\ce-$(quasi-)valuation is an $\ce'-$(quasi-)valuation and any $\ce'-$characteristic function in $\ce$ is an $\ce-$valuation. 
$\ce^{max}-$valuations and
$\ce^{\mathrm{add}}-$characteristic functions in $\ce^{\max}$ were called respectively \emph{admissible} and \emph{strongly admissible} in \cite{FFR1}.

\begin{proposition} \label{filtrations_and_char}
\begin{itemize}
\item[(i)]If $w$ is an $\ce-$valuation, there exists an algebra filtration $\cF^w_{\bullet}$ on $\cH(\cE)$
where $\cF^w_n$ is spanned by the $[M]$ such that $w(M) \leq n$, which is normalized if $w$ is so.

\item[(ii)] If $w$ is $\ce'-$characteristic in $\ce$, the associated graded algebra of $\cH(\cE)$ with respect to this
filtration is $\ch(\ce')$.
\end{itemize}
\end{proposition}

\begin{proof}
The proof of \cite[Corollary 1]{FFR1} applies here word for word. It uses only \cite[Lemma 3]{FFR1}. The idea in (i) is that the basis $\left\{[A] \ | A \in \Iso(\ca) \right\}$ is \emph{compatible} with the filtration $\cF^w_{\bullet}$ in the sense that its intersection with $\cF^w_n$ is a basis of  $\cF^w_n$ for all $n.$ Part (ii) then follows from the form of structure constants in the definitions of $\ch(\ce)$ and $\ch(\ce')$, as well as the definition of  $\ce'-$characteristic functions in $\ce$.
\end{proof}

\begin{corollary}\label{Cor:q-symm}
If $w$ is $\ce^{\add}-$characteristic in $\ce$, the associated graded algebra of $\cH(\cE)$ with respect to the
filtration $\cF^w_{\bullet}$ is $\ch(\ce^{\add}).$ It is a skew polynomial algebra in variables in $\Iso(\cA),$ namely, the monoid algebra of $(\Iso(\cA), \oplus),$ twisted by the form $\frac{1}{|\Hom_{\ca}(-,-)|}$.
\end{corollary}

\begin{proposition} \label{each_val_is_char}
Let $\cA$ be a $\Hom-$finite idempotent complete $\mathbb{F}_q-$linear additive category and let $\cE$ be an $\Ext^1-$finite exact structure on $\cA.$ Then for each $\cE-$valuation $w: \Iso(\cA) \to \mathbb{N},$ there exists $\cE' \in [\ce^{\add}, \ce]$ such that $w$ is $\cE'$-characteristic in $\cE.$
\end{proposition}

\begin{proof}
Define a function $\widetilde{w}: \Iso (\eff \cE) \to \mathbb{N}$ by the rule
$$\widetilde{w}(\delta^{\sharp}) := w(A) - w(B) + w(C),$$
whenever there exists a conflation $\delta$ in $\ce$ of the form 
$$A \rightarrowtail B \twoheadrightarrow C.$$
Recall that
$$0 \to \Hom(-, A) \overset{\Hom(-, f)}\longrightarrow \Hom(-, B) \overset{\Hom(-, g)}\longrightarrow \Hom(-, C) \to \delta^{\sharp}  \to 0$$ is a projective resolution of $\delta^{\sharp}.$ By the Yoneda lemma combined with Schanuel's lemma, the function $\widetilde{w}(\delta^{\sharp})$   well-defined. That means that if there is another conflation $\gamma$ of the form 
$$A' \rightarrowtail B' \twoheadrightarrow C',$$
such that $\delta^{\sharp} \overset\sim\to \gamma^{\sharp},$ then 
$w(A) - w(B) + w(C) = w(A') - w(B') + w(C'),$
so the alternating sum defining $\widetilde{w}$ does not depend on the choice of conflation.
By the Horseshoe lemma, we have 
$$\widetilde{F} + \widetilde{H} = \widetilde{G}$$ for each short exact sequence 
$$0 \to F \to G \to 0$$
in $\eff \ce.$
Then the kernel of $\widetilde{w}$ is a Serre subcategory in $\eff \cE$. Then by Theorem \ref{classif_general:body}, it canonically defines an exact structure $\cE' \in [\ce^{\add}, \ce]$ on $\cA$: a conflation $\mu$  in $\ce$ is also a conflation in $\ce'$ if and only if $\widetilde{w}(\mu^{\sharp})=0.$
By definition of $\ce',$ $w$ is $\ce'$-characteristic in $\ce.$
\end{proof}

\begin{theorem} \label{each_degen_is_HA:body}
Let $\cA$ be a $\Hom-$finite idempotent complete $\mathbb{F}_q-$linear additive category and let $\cE$ be an $\Ext^1-$finite exact structure on $\cA.$ Then for each $\cE-$valuation $w: \Iso(\cA) \to \mathbb{N},$ the associated graded of the filtration induced by $w$ on the Hall algebra $\cH(\cE)$ is the Hall algebra $\cH(\cE'),$ for some $\cE' \in [\ce^{\add}, \ce].$
\end{theorem}

\begin{proof}
This follows directly from Proposition \ref{each_val_is_char} combined with Proposition \ref{filtrations_and_char}.
\end{proof}

\subsection{Admissible case}

For the categories of representations of Dynkin quivers, all $\ce^{\max}-$valuations and all $\ce^{\mathrm{add}}-$characteristic functions in $\ce^{\max}$ were classified in \cite{FFR1}. The proof uses only the existence of AR conflations in $\ce^{\max}$ and the property $\Ex(\ce^{\max}) = \AR(\ce^{\max})$ for these categories. With minor changes, we can generalize this classification to all exact structures on locally finite categories. If we restrict ourselves only to functions $w$ satisfying certain additional conditions, we prove similar results even for admissible exact structures beyond the locally finite case. For $X\in\cA$, we set $w_X:\Iso(\cA)\to\mathbb{N}$ to be the function $w_X(M):= \dim \Hom_\cA(X, M)$.

\begin{lemma} (cf. \cite[Theorem 4 (1)]{FFR1}) \label{max_adm}

Let $\cA$ be a $\Hom-$finite idempotent complete $k-$linear additive category. 
\begin{itemize}
\item[(i)] 
The function $w_X$ is an $\ce^{\max}-$valuation, for each $X \in \cA.$ If $X \in \cp(\cE),$ then $w_X$ is $\ce-$additive.
\item[(ii)]
Suppose that for $X \in \ind(\ca) \backslash \ind(\cp(\ce^{\max})),$ there exists an AR conflation in $\ce^{\max}$ ending at $X$:
\begin{align} \label{AR:thm_max_adm}
\tau X  \rightarrowtail T \twoheadrightarrow X.
\end{align}
Then for each $Y \in \ind(\cA),$ we have
$$w_Y(X) - w_Y(T) + w_Y(\tau X) = \delta_X^Y.$$
Dually, suppose that there exists an AR conflation in $\ce^{\max}$ starting at $X$:
\begin{align} \label{AR:thm_max_adm2}
X  \rightarrowtail U \twoheadrightarrow \tau^{-1} X.
\end{align}
Then for each $Y \in \ind(\cA),$ we have
$$w_Y(X) - w_Y(T) + w_Y(\tau^{-1} X) = \delta_X^Y.$$
\item[(iii)]
Any linear combination $\sum\limits_{V \in \ind(\ca)} a_V w_V$ with $a_V \in \mathbb{Z} \forall V \in \ind(\ca)$ is additive, whenever it is well-defined. If all $a_V$ are non-negative, it is an $\ce^{\max}-$valuation.
If $\ca$ is locally finite, each such formal sum is well-defined.
\end{itemize}
\end{lemma}

\begin{proof}
Part (i) follows from the existence of long exact sequences. The alternating sum 
$w_Y(X) - w_Y(T) + w_Y(\tau X)$ in (ii) is equal to $\dim(S_X)$, since $S_X$ is the contravariant defect of the conflation (\ref{AR:thm_max_adm}). Dually, $w_Y(X) - w_Y(T) + w_Y(\tau^{-1} X)$ equals $\dim(S^X)$. Part (ii) then follows from Remark \ref{supp_simple}.

In (iii), the fact that the sum $\sum\limits_{V \in \ind(\ca)} a_V w_V$ is additive whenever it is well-defined follows from (i).
When all $a_V$ are non-negative, the sum is an $\ce^{\max}-$valuation again by (i). If $\ca$ is locally finite, $\supp(\Hom(-, X))$ is finite for any object $X,$ and so only finitely many terms contribute to the sum. This means the sum is well-defined for any choice of $a_V.$
\end{proof}

\begin{theorem} \label{degen:body}
Let $\cA$ be a $\Hom-$finite idempotent complete $k-$linear additive category and $\ce$ be an admissible exact structure on $\cA$. Let $\cE'<\cE$ be a smaller exact structure. Suppose 
\begin{align} \label{supp}
\supp \Hom(-, M)|_{\cp(\ce') \backslash(\cp(\ce))} < \infty, \quad \forall M \in \cA.
\end{align}
Then the function
\begin{equation} \label{weight_function_relative}
w_{\ce, \ce'}:  \Iso(\cA) \to \mathbb{N}, \quad\quad w_{\cE,\cE'} = \sum\limits_{P \in \ind(\cp(\ce'))\backslash \ind(\cp(\ce))} w_P
\end{equation}
is well-defined and $\ce'-$characteristic in $\ce$. It defines an algebra filtration on $\ch(\ce)$ whose associated graded is $\ch(\ce').$
\end{theorem}

\begin{proof}
By the condition \eqref{supp}, for each $M \in \cA$, only finitely many functions in the formal sum on the right hand side of \eqref{weight_function_relative} take non-zero value on $M$. Since $\ca$ is $\Hom-$finite, these values are well-defined (i.e. finite). Therefore, the function $w_{\ce, \ce'}$ is well-defined. 
It is clearly $\ce'-$additive. 
Since each term is an $\ce^{\max}-$valuation by Lemma \ref{max_adm}(i), each term is a $\ce-$valuation is well, and then so is the sum $w_{\ce, \ce'}$.

Let $ X \rightarrowtail Y \twoheadrightarrow Z$ be a conflation in $\ce.$ By Theorem \ref{Thm:Eno2}, we have a decomposition
$$[X] - [Y] + [Z] = \sum\limits_{U \in \ind(\ca) \backslash \ind(\cp(\ce))} ([\tau U] - [C] + [U]) b_U,$$
for some $b_U \in \mathbb{N}$, where 
$$\tau U  \rightarrowtail C \twoheadrightarrow U$$
is the AR conflation in $\ce$ ending at $U.$

Then we have
\begin{eqnarray*}
w_{\ce, \ce'}([X] - [Y] + [Z]) &=& \sum\limits_{U \in \ind(\ca) \backslash \ind(\cp(\ce))} w_{\ce, \ce'}([\tau U] - [C] + [U]) b_U \\
&=& \sum\limits_{\substack{U \in \ind(\ca) \backslash \ind(\cp(\ce)) \\ P \in \ind(\cp(\ce'))) \backslash \ind(\cp(\ce)}} w_P ([\tau U] - [C] + [U]) b_U \\
&=& \sum\limits_{\substack{U \in \ind(\ca) \backslash \ind(\cp(\ce)) \\  P \in \ind(\cp(\ce')) \backslash \ind(\cp(\ce))}} b_U \delta_P^U \\
&=& \sum\limits_{P \in  \ind(\cp(\ce')) \backslash \ind(\cp(\ce))} b_P.
\end{eqnarray*}

Since all $b_P \geq 0,$ we have 
$$w_{\ce,\cE'}([X] - [Y] + [Z]) = 0$$
if and only if for any $P \in  \ind(\cp(\ce')) \backslash \ind(\cp(\ce))$, $b_P = 0$.
This happens if and only if we have a decomposition 
$$[X] - [Y] + [Z] = \sum\limits_{U \in \ind(\ca) \backslash \ind(\cp(\ce'))} ([\tau U] - [C] + [U]) b_U,$$
for some $b_U \in \mathbb{N}$ and such that at least one of $b_U$ is non-zero. Applying Theorem \ref{Thm:Eno2}, we find that this happens if and only if the conflation 
$$ X \rightarrowtail Y \twoheadrightarrow Z$$
is a conflation in $\cE'.$ Then if this conflation is in $\ce \backslash \ce'$  instead, we have
$$w_{\ce, \ce'}([X] - [Y] + [Z]) > 0,$$
or, equivalently,
$$w_{\ce, \ce'}(Y) < w_{\ce, \ce'}(X) + w_{\ce, \ce'}(Z).$$
This proves that $w_{\ce, \ce'}$ is $\ce'-$characteristic in $\ce$. The rest follows from Proposition \ref{filtrations_and_char}.
\end{proof}

By our alternative proof of Theorem \ref{ZZ}, if $\cA$ is locally finite, then any exact structure on it is admissible. By definition of local finiteness, condition (\ref{supp}) is satisfied for all pairs $\cE' < \ce$ of exact structures on $\cA.$ Therefore, the Theorem \ref{degen:body} applies to all pairs $\ce' < \ce$ and we get the following corollaries.

\begin{corollary}\label{Cor:DegenHall:locfin}
Let $\cA$ be a $\Hom-$finite, locally finite idempotent complete $k-$linear additive category. 
\begin{itemize}
\item 
Let $\ce$ be an exact structure on $\cA.$ Then the function
\begin{equation} \label{weight_function_absolute}
w_{\ce^{\max},\ce}:  \Iso(\cA) \to \mathbb{N}, \quad\quad w_\cE = \sum\limits_{P \in \ind(\cp(\ce)) \backslash \ind(\cp(\ce^{\max}))} w_P
\end{equation}
is well-defined and $\ce-$characteristic in $\ce^{\max}$. 
\item
Let $\ce$ be an exact structure on $\cA.$ Then the function
\begin{equation} 
w_{\ce, \ce^{\add}}:  \Iso(\cA) \to \mathbb{N}, \quad\quad w_{\ce, \ce^{\add}} = \sum\limits_{P \in \ind(\ca) \backslash \ind(\cp(\ce)} w_P
\end{equation}
is well-defined and $\ce^{\add}-$characteristic in $\ce$. 
\item 
For any pair of exact structures $\ce' < \ce,$ the function $w_{\ce,\ce'}$ defines an algebra filtration on $\cH(\cE)$ whose associated graded algebra is 
$\ch(\ce')$.
\end{itemize}
\end{corollary}

\begin{corollary} \label{converse_statement1}
Let $\cA$ be a $\Hom-$finite, locally finite idempotent complete $\mathbb{F}_q-$linear additive category. For every set of pairwise non-isomorphic indecomposables $\ci \subset \Ind(\cA)$, for every set $\left\{a_V | V \in \ci \right\}$ of natural numbers ($a_V \in \mathbb{N} \quad \forall V \in \ci$), the function 
$$\sum_{V \in \ci} a_V w_V$$ 
is well-defined and induces an algebra filtration on $\cH(\cE^{\max})$. The associated graded algebra is
$\cH(\cE),$ where $\cE$ is the unique exact structure whose set of projectives is $\ci \cup \ind(\cp(\cE^{\max})).$ 
\end{corollary}

\subsection{Comparison with the approach of Berenstein-Greenstein} \label{BG}

In the work \cite{BG}, Berenstein and Greenstein proved a PBW-type theorem for Hall algebras associated to $\Hom-$ and $\Ext^1-$finite exact categories. To be more precise, they showed that the Hall algebra of such an exact category $\ce$ is generated by the classes of indecomposable objects. As in the PBW-theorem for universal enveloping algebras, the crucial point is to define an algebra filtration such that the associated graded algebra has a simple structure. In the situation of Hall algebras, there is no natural candidate for such a filtration. 

For this they considered a proset (preordered set) structure $\triangleleft$ on $\Iso(\cA)$ (recall that we have fixed an additive category $\mathcal{A}$) by requiring for any nonsplit short exact sequence $X \rightarrowtail Y \twoheadrightarrow Z$, $[Y]\triangleleft[X\oplus Z]$ and taking the transitive closure. A function from $\Iso(\cA)$ to $\mathbb{N}$ preserving this preorder, when it exists, gives a filtration on the Hall algebra of $\ce$. We denote $\mathcal{F}$ the set of such functions, which coincides with the set of $\cE$-quasi-valuations. Whether the set $\mathcal{F}$ is empty, or how to explicitly construct such a function, is \emph{a priori} not clear. 

The important observation in \cite[Proposition 4.8]{BG} is that 
$$f:\Iso(\cA) \to \mathbb{N},\ \  [M] \mapsto  |\mathrm{End}(M)|$$ 
is an $\cE$-quasi-valuation on $\cH(\cE)$. It is in general not an $\cE$-valuation: for two arbitrary objects $[M],[N]\in\Iso(\cA)$, we can not expect the equality in $f([M\oplus N])=f([M])+f([N])$ to hold. 

In this paper we considered a subset of $\mathcal{F}$ consisting of $\cE^{\add}$-additive functions, which we call $\cE$-valuations. Such functions can be classified using Auslander-Reiten theory. It is interesting to pursue that which kind of structures does $\mathcal{F}$ admit, as well as whether one can approximate it by $\cE^{\add}$-additive functions.

In the general setting, for two exact structures $\cE' < \cE$, we can look at functions $f:\Iso(\cA)\to\mathbb{N}$ satisfying $f(D)<f(B)$ for any two conflations 
$$A \rightarrowtail  B \twoheadrightarrow C\in \cE',\ \   A \rightarrowtail D \twoheadrightarrow C\in \cE\backslash \cE'.$$
Such functions will produce filtrations to degenerate from $\mathcal{H}(\cE)$ to $\mathcal{H}(\cE')$. The functions studied in \cite{BG} are special cases when $\cE' = \cE^{\add}$.

\subsection{Degenerations and twists}

Recall that the nilpotent part of a quantum group of a finite valued quiver $Q$ is isomorphic to (a subalgebra of) the twisted Hall algebra $\ch(\mod k Q)$ rather than to its untwisted counterpart. Therefore, it makes sense to investigate the effect of changes of exact structures on twisted Hall algebras. We begin with the following simple observation.

\begin{lemma}
Let $\cE' < \cE$ be a pair of exact structures on an additive category $\cA$. If the Euler form of $\cE$ is well-defined, it descends to a bilinear form on the Grothendieck $K_0(\mathcal{E'})$ of $\mathcal{E'}$,
denoted by the same symbol:
$$\left\langle \cdot, \cdot \right\rangle_{\cE}: K_0(\mathcal{E'}) \times K_0(\mathcal{E'}) \to \mathbb{Q}^{\times}.$$
\end{lemma}

\begin{proof}
Since each conflation in $\cE'$ is a conflation in $\cE,$ the statement follows from the existence of long exact sequences of $\Ext^i-$functors in $\cE.$
\end{proof}

As explained in Section 3, this implies that the Euler form $\left\langle \cdot, \cdot \right\rangle_{\cE}$ and its square root $\sqrt{\left\langle \cdot, \cdot \right\rangle_{\cE}}$ give well-defined deformations of the Hall algebra $\ch(\ce').$ We get the following reformulations of main results of Sections 4.1 and 4.2.

\begin{theorem}
Let $\cA$ be a $\Hom-$finite, idempotent complete and $k-$linear additive category. Let $\cE' < \cE$ be a pair of $\Ext^1-$finite admissible exact structures on $\cA.$ 
\begin{itemize}
    \item[(i)] If $w: \Iso(\cA) \to \mathbb{N}$ is an $\cE-$quasi-valuation, there exists a filtration $\cF_{\bullet}$ on $\cH_{tw}(\cE)$ where $\cF_n$ is spanned by the $[M]$ such that $w(M) \leq n,$ which is normalized if $w$ is so.
    \item[(ii)] If $w$ is $\ce'-$characteristic, the associated graded of $\cH_{tw}(\cE)$ with respect to this filtration is $\cH(\cE')$ twisted by the form $\sqrt{\left\langle \cdot, \cdot \right\rangle_{\cE}}.$ That is, the product of two classes $[A]$ and $[B]$ is given by the formula
    $$\sqrt{\left\langle A, B \right\rangle_{\cE}} \sum\limits_{C \in \Iso(\cA)} \frac{|\Ext_{\ce'}^1(A,B)_C|}{|\Hom(A, B)|} [C].$$
    \item[(iii)] The function $w_{\ce, \ce'}$ given by the formula (\ref{weight_function_relative}) is $\cE'$-characteristic and induces a filtration from item (ii).
    \item[(iv)] Any filtration on $\cH_{tw}(\cE)$ given by an $\cE-$valuation is the Hall algebra of a smaller exact structure, twisted by the form $\sqrt{\left\langle \cdot, \cdot \right\rangle_{\cE}}.$
\end{itemize}
\end{theorem}

\begin{corollary} \label{Cor:DegenHall:locfin:twist}
Let $\cA$ be a $\Hom-$finite, locally finite, idempotent complete and $k-$linear additive category. Let $\cE$ be an $\Ext^1-$finite exact structure on $\cA.$ Then the function $w_{\ce, \ce^{\add}}$ defined in Corollary \ref{Cor:DegenHall:locfin} defines an algebra filtration on $\ch_{tw}(\ce)$ whose associated graded is  $\cH(\cE^{\add})$ twisted by the form $\sqrt{\left\langle \cdot, \cdot \right\rangle_{\cE}}.$  That is, the product of two classes $[A]$ and $[B]$ is given by the formula
    $$\frac{\sqrt{\left\langle A, B \right\rangle_{\cE}}} {|\Hom(A, B)|} [A \oplus B].$$
\end{corollary}

If we have two hereditary exact categories and a bijection $F$ on the set of objects such that $| \Ext^1(A, B)_C| = |\Ext^1(FA, FB)_{FC}|,$ their Hall algebras are related by a well-defined twist
$$\frac{|\Hom(A, B)|}{|\Hom(FA, FB)|}$$ (these categories have isomorphic Grothendieck groups and this quotient descends from the sets of isomorphism classes to the Grothendieck groups). 

In particular, if we take $\cE, \cE'$ exact structures on the same hereditary category $\cA,$ with $\cE$ and $\cE'$ having only one AR sequence each: $A \rightarrowtail B \twoheadrightarrow C$ and $A' \rightarrowtail B' \twoheadrightarrow C',$
with all the terms being indecomposable, then $\cH(\cE)$ and $\cH(\cE')$ are related by a twist. In principle, even for a pair of different exact structures $\ce, \ce',$ this twist can be trivial. In this case, $\cH(\cE)$ and $\cH(\cE')$ are isomorphic. We discuss the minimal such example in Example \ref{disjoint_A2}.

For any pair of additive categories  with the same cardinality of the set of indecomposables, Hall algebras of their split exact structures are just $q$-symmetric polynomials on the same set of variables, but the powers of $q$ appearing in the commutation relations may be completely different. 

Under conditions of Corollary \ref{Cor:DegenHall:locfin:twist}, the associated graded of $\ch_{tw}(\ce)$ with respect to the filtration given by $w_{\ce, \ce^{\add}}$ is always an algebra of  $q$-symmetric polynomials in variables $\left\{[A] | A \in \Ind(\ca)\right\}$, but the powers of $q$ appearing in the commutation relations depend on the structure $\ce.$

\section{Polyhedral cones and degenerations of Hall algebras}\label{Sec:5}

Let $\cA$ be a $\Hom-$finite idempotent complete $k-$linear additively finite category and $\cE'<\cE$ be a pair of exact structures on $\cA$. For an exact structure $\cE$ on $\cA$, we denote $\mathcal{H}(\cE)$ the corresponding Hall algebra. Following \cite{FFR2} we define two polyhedral cones from these data. 

The first one is called the \emph{$K$-theoretic cone} of such pair of exact structures. 
Since $\cE'<\cE$, there exists a canonical surjection $K_0(\cE')\twoheadrightarrow K_0(\cE)$. We let $\Lambda^{\cE,\cE'}$ denote the kernel of this surjection. Since $K_0(\cE')$ is a free abelian group, so is $\Lambda^{\cE,\cE'}$. We set 
$$\Lambda^{\cE,\cE'}_{\mathbb{R}}:=\Lambda^{\cE,\cE'}\otimes_\mathbb{Z}\mathbb{R}.$$
Let $\cC^{\cE,\cE'}\subseteq\Lambda^{\cE,\cE'}_{\mathbb{R}}$ be the polyhedral cone generated by $[X]-[Y]+[Z]$ for $X \rightarrowtail Y \twoheadrightarrow Z$ a conflation in $\cE\setminus\cE'$. We set $\cC^\cE:=\cC^{\cE,\cE^{\mathrm{add}}}$. 
As a consequence of Theorem \ref{enomoto}, we have:

\begin{proposition}\label{Simplicial}
The extremal rays of the polyhedral cone $\cC^{\cE,\cE'}$ are given by $[X]-[Y]+[Z]$ for $X \rightarrowtail Y \twoheadrightarrow Z$ an AR-conflation in $\cE\setminus\cE'$, hence the cone $\cC^{\cE,\cE'}$ is simplicial.
\end{proposition}

\begin{proof}
According to Theorem \ref{enomoto} (v), the polyhedral cone $\cC^{\cE,\cE'}$ is generated by  $[X]-[Y]+[Z]$ for $X \rightarrowtail Y \twoheadrightarrow Z$ an AR-conflation in $\cE\setminus\cE'$. In order to show that these are extremal rays, it suffices to notice that since $\cA$ is additively finite, any non-projective indecomposable object appears in one and only one AR-conflation as the rightmost term.
\end{proof}

We look at the extremal case. Let $\Lambda:=\Lambda^{\cE^{\max},\cE^{\mathrm{add}}}$, $\Lambda_\mathbb{R}:=\Lambda^{\cE^{\max},\cE^{\mathrm{add}}}_\mathbb{R}$ and $\cC:=\cC^{\cE^{\max},\cE^{\mathrm{add}}}$ be the cone in $\Lambda_\mathbb{R}$. The polyhedral cone $\cC$ is simplicial of dimension $N-n$ where $N$ is the rank of $K_0^{\mathrm{add}}(\cA)$ and $n$ is the cardinality of $\mathrm{ind}(\mathcal{P}(\cE^{\max}))$. 

\begin{remark}\label{Rmk:Enomoto}
The fact that the exact structures on $\cA$ form a Boolean lattice (\cite{Enomoto1}, see also \cite[Theorem 5.7]{BHLR}) can be rephrased into: the face lattice of the simplicial cone $\cC$ is a Boolean lattice on $N-n$ elements.
\end{remark}

The second cone is called the \emph{degree cone} of the Hall algebra $\mathcal{H}(\cE)$. For simplicity we denote $\mathcal{I}:=\mathrm{ind}(\cA)$. By \cite[Theorem 1.1]{BG}, the Hall algebra $\mathcal{H}(\cE)$ is generated by $[M]$ for $M\in\mathcal{I}$. Any function $\mathbf{d}:\mathcal{I}\to\mathbb{R}$ gives rise to a filtered vector space structure on $\mathcal{H}(\cE)$ by imposing $\mathbf{d}(M)$ as the degree of $[M]$. For a pair of exact structures $\cE'<\cE$, the associated degree cone is defined by:
$$\mathcal{D}^{\cE,\cE'}:=\{\mathbf{d}\in\mathbb{R}^\mathcal{I}\mid \mathbf{d}\text{ induces an algebra filtration and }\mathrm{gr}_\mathbf{d}(\mathcal{H}(\cE)) = \mathcal{H}(\cE')\},$$
where $\mathrm{gr}_\mathbf{d}$ stands for taking the associated graded algebra with respect to the filtration induced by $\mathbf{d}$. We set $\mathcal{D}^\cE:=\mathcal{D}^{\cE,\cE^{\mathrm{add}}}$ and $\mathcal{D}:=\mathcal{D}^{\cE^{\max}}$. The closure $\overline{\mathcal{D}^{\cE,\cE'}}$, when being non-empty, is a polyhedral cone.

\begin{theorem}\label{Thm:cone}
We have
$$\mathcal{D}^{\cE,\cE'}=\{\varphi\in (K_0^{\mathrm{add}}(\cA)\otimes_\mathbb{Z}\mathbb{R})^*\mid \text{for any }x\in\cC^{\cE,\cE'}, \varphi(x)>0;\ \text{for any }y\in\cC^{\cE'},\ \varphi(y)=0\}.$$
\end{theorem}

Since $\mathrm{ind}(\cA)$ is a $\mathbb{Z}$-basis of $K_0^{\mathrm{add}}(\cA)$, both sides of the above identity are in fact in the same space.

\begin{proof}
For a function $\varphi$ from the right hand side and a conflation $X \rightarrowtail Y \twoheadrightarrow Z\in\cE\setminus\cE'$, by Theorem \ref{enomoto} (v), we can write $[X]-[Y]+[Z]$ as a non-negative sum of AR-conflations in $\cE$. There exists an AR-conflation $A \rightarrowtail B \twoheadrightarrow C$ not in $\cE'$ having non-zero coefficient. This implies $\varphi(X)+\varphi(Z)>\varphi(Y)$. Moreover, if such a conflation $X \rightarrowtail Y \twoheadrightarrow Z\in\cE'$ then $\varphi(X)+\varphi(Z)=\varphi(Y)$. This implies that $\varphi$ induces an algebra filtration on $\mathcal{H}(\cE)$ whose associated graded algebra is $\mathcal{H}(\cE')$.

Directly translating the properties in the definition of $\mathcal{D}^{\cE,\cE'}$ into conditions on extensions shows the other inclusion.
\end{proof}

When projected to the corresponding subspace, the cone $\overline{\mathcal{D}^{\cE,\cE'}}$ can be naturally identified with the dual cone of $\mathcal{C}^{\cE,\cE'}$ in $(\Lambda_\mathbb{R}^{\cE,\cE'})^*$.

\section{Examples}\label{Examples}

Let $Q$ be a finite Dynkin quiver and $\mathcal{A}$ be the additive category of finite dimensional representations of $Q$ over $\mathbb{F}_q$. For an exact structure $\cE$ on $\cA$, we denote $\mathcal{H}(Q,\cE)$ the corresponding Hall algebra. 

\begin{example}
For $\ce=\ce^{\max}$ and $\ce'=\ce^{\mathrm{add}}$, we recover the result in \cite{FFR2}. Let $\underline{w}_0$ be a reduced decomposition of the longest element in the Weyl group of the Lie algebra $\mathfrak{g}(Q)$ associated to the underlying graph of $Q$. We assume furthermore that $\underline{w}_0$ is compatible with the orientation of $Q$.

\begin{proposition}[\cite{FFR2}]
The polyhedral cone $\overline{\mathcal{D}}$ coincides with the quantum degree cone associated to $\underline{w}_0$ in \cite{BFF}. Its points induce degenerations of the negative part of the quantum group $U_\nu^-(\mathfrak{g}(Q))$.
\end{proposition}
\end{example}

\begin{example}\label{Ex:An}
We consider the quiver
$$Q=1\longrightarrow 2\longrightarrow \cdots\longrightarrow n.$$ 
For $1\leq i\leq j\leq n$, let $I_{i,j}$ denote the indecomposable representation of $Q$ supported at nodes $i,i+1,\cdots,j$ and $S_i:=I_{i,i}$. The AR-conflations in $\cE^{\max}$ are
$$(A).\ \ I_{i,j}\to I_{i-1,j}\oplus I_{i,j-1}\to I_{i-1,j-1},\  \text{ for }2\leq i<j\leq n;$$
$$(B).\ \ S_{i+1}\to I_{i,i+1}\to S_i,\  \text{ for }1\leq i\leq n-1.$$
This example is studied in \cite{FFR1,FFR2,FFFM} under the framework of PBW filtration and the maximal cone in the tropical flag variety corresponding to the Feigin-Fourier-Littelmann-Vinberg (FFLV) toric variety.

We briefly discuss this connection in the above framework, complete proofs can be found in \cite{CFFFR,FFR1,FFR2,FFFM}.

Let $\mathfrak{g}=\mathfrak{sl}_{n+1}$ be the type $\tt A_n$ Lie algebra and $\mathfrak{g}=\mathfrak{n}^+\oplus\mathfrak{h}\oplus\mathfrak{n}^-$ be a triangular decomposition of $\mathfrak{g}$. The positive roots in $\mathfrak{g}$ will be denoted by $\Delta_+:=\{\alpha_{i,j}\mid 1\leq i\leq j\leq n\}$. For a positive root $\alpha_{i,j}$ in $\mathfrak{g}$, we fix a generator $f_{\alpha_{i,j}}$ in the weight space of $\mathfrak{n}^-$ of weight $-\alpha_{i,j}$.

Let $\cE_{\mathrm{PBW}}$ be the smallest exact structure on $\cA$ containing the AR-conflations $S_{i+1}\to I_{i,i+1}\to S_i$ for $1\leq i\leq n-1$. Let $\mathcal{D}_{\mathrm{PBW}}$ be the intersection of the degree cone $\overline{\mathcal{D}^{\cE_{\mathrm{PBW}}}}$ (recall that $\mathcal{D}^\cE:=\mathcal{D}^{\cE,\cE^{\mathrm{add}}}$) with the face of $\overline{\mathcal{D}}$ defined by those $\varphi$ taking zero value on the AR-conflations above of type (A). The polyhedral cone $\mathcal{D}_{\mathrm{PBW}}$ is simplicial of dimension $n-1$ having $2^{n-1}$ faces. 

By Corollary \ref{converse_statement1}, the Hall algebra $\mathcal{H}(\cE_{\mathrm{PBW}})$, up to a twist, is a degeneration of the quantum group $U_\nu^-(\mathfrak{g})$. Any point $\mathbf{d}$ in the cone $\mathcal{D}_{\mathrm{PBW}}$ induces a filtration on $\mathcal{H}(\cE_{\mathrm{PBW}})$. According to the PBW-type theorem for quantum groups in Lusztig \cite{Lusztig} and Ringel \cite{R1} (see also the arguments in \cite[Section 2.2 and 2.3]{FFR1}), by taking the $\nu\mapsto 1$ limit, the Hall algebra $\mathcal{H}(\cE_{\mathrm{PBW}})$ can be specialized to the universal enveloping algebra $U(\mathfrak{n}^-_\mathbf{d})$ of some Lie algebra $\mathfrak{n}^-_\mathbf{d}$, which is isomorphic to the associated graded algebra of the universal enveloping algebra $U(\mathfrak{n}^-)$ with respect to the filtration induced by giving degree degree $\mathbf{d}(I_{i,j})$ to the PBW root vector $f_{\alpha_{i,j}}\in \mathfrak{n}^-$.

Let $V_\lambda$ be the finite dimensional irreducible representation of $\mathfrak{g}$ of highest weight $\lambda$ and $v_\lambda\in V_\lambda$ be a highest weight vector. For $k\geq 0$, let $\mathcal{F}_k U(\mathfrak{n}^-)$ denote the linear span of elements in $U(\mathfrak{n}^-)$ having degrees no more than $k$. It induces a filtration on $V_\lambda$ by setting $\mathcal{F}_k V_\lambda:=\mathcal{F}_k U(\mathfrak{n}^-)\cdot v_\lambda$. We denote $V_\lambda^\mathbf{d}$ the associated graded vector space: it admits a cyclic $U(\mathfrak{n}^-_\mathbf{d})$-module structure. Let $v_\lambda^\mathbf{d}$ be the class of $v_\lambda$ in $V_\lambda^\mathbf{d}$.

For example, we take $\mathbf{p}\in\overline{\mathcal{D}}$ such that for any $1\leq i\leq j\leq n$, $\mathbf{p}(I_{i,j})=1$: it is in $\mathcal{D}_{\mathrm{PBW}}$. The filtration on $U(\mathfrak{n}^-)$ induced by $\mathbf{p}$ is the usual PBW filtration used in the PBW theorem. Although the Hall algebra $\mathcal{H}(\cE_{\mathrm{PBW}})$ is non-commutative, its $\nu\mapsto 1$ specialization $U(\mathfrak{n}_\mathbf{p}^-)$ is isomorphic to the symmetric algebra $S(\mathfrak{n}^-)$. Studying a monomial basis of $V_\lambda^\mathbf{p}$ and its polyhedral parametrization is the main topic in \cite{FFL}.

For any $\mathbf{d}\in \mathcal{D}_{\mathrm{PBW}}$, the Lie algebra $\mathfrak{n}_\mathbf{d}^-$ appeared in \cite{CFFFR}, via the exponential map, as the group of symmetries of some linear degenerate flag varieties. We make its connection to the face lattice of $\mathcal{D}_{\mathrm{PBW}}$ explicit.

According to Proposition \ref{Simplicial} and Theorem \ref{Thm:cone}, the face lattice of $\mathcal{D}_{\mathrm{PBW}}$ is isomorphic to the canonical Boolean lattice of the subsets of $[n-1]=\{1,2,\cdots,n-1\}$. Precisely, for any subset $I:=\{i_1,\cdots,i_k\}$ of $[n-1]$, there exists a unique face $F_I$ of $\mathcal{D}_{\mathrm{PBW}}$ consisting of those $\phi$ satisfying 
\begin{enumerate}
    \item for $i\in I$, $\phi([S_i])+\phi([S_{i+1}])>\phi([I_{i,i+1}])$;
    \item for $j\in [n-1]\setminus I$,
    $\phi([S_j])+\phi([S_{j+1}])=\phi([I_{j,j+1}])$.
\end{enumerate}
All faces of $\mathcal{D}_{\mathrm{PBW}}$ arise in this way.

Let $\rho$ be the half sum of all positive roots. We associate to $\mathbf{d}\in \mathcal{D}_{\mathrm{PBW}}$ a geometric object 
$$\mathcal{F}l_{n+1}^\mathbf{d}:=\overline{\exp(\mathfrak{n}_\mathbf{d}^-)\cdot [v_\rho^\mathbf{d}]}\subseteq \mathbb{P}(V_\rho^{\mathbf{d}})$$
as the highest weight orbit in the representation $V_\rho^\mathbf{d}$ through the cyclic vector $v_\rho^\mathbf{d}$. This is a flat degeneration of the usual complete flag variety. It is easy to see that if $\mathbf{d}$ and $\mathbf{d}'$ lie in the relative interior of the same cone in $\mathcal{D}_{\mathrm{PBW}}$, $\mathcal{F}l_{n+1}^{\mathbf{d}}$ is isomorphic to $\mathcal{F}l_{n+1}^{\mathbf{d}'}$ as projective varieties. This procedure associate a family of isomorphic varieties to every element in the face lattice of $\mathcal{D}_{\mathrm{PBW}}$. 

These varieties appear naturally in the framework of linear degenerate flag varieties, where the poset structure on the face lattice is interpreted as the orbit closure structure of degeneration of representations. Let $\mathbf{e}:=(1,2,\cdots,n)$ and $M$ be a representation of the quiver $Q$ of dimension vector $(n+1,\cdots,n+1)$. The quiver Grassmannian $\mathrm{Gr}_\mathbf{e}(M)$ consisting of subrepresentations of $M$ having dimension vector $\mathbf{e}$ is called a linear degenerate flag variety in \cite{CFFFR} since when $M_0=I_{1,n}^{\oplus n+1}$, $\mathrm{Gr}_\mathbf{e}(M_0)$ is isomorphic to the complete flag variety. 

For any $\mathbf{d}\in \mathcal{D}_{\mathrm{PBW}}$, we let $I_\mathbf{d}=\{i_1,\cdots,i_k\}\subseteq [n-1]$ be such that $\mathbf{d}\in F_{I_\mathbf{d}}$. We fix the ordering $i_1<\cdots<i_k$ and consider the following representation of $Q$:
$$M_\mathbf{d}:=I_{1,n}^{\oplus n+1-k}\oplus\bigoplus_{\ell=1}^k (I_{1,i_\ell}\oplus I_{i_\ell+1,n}).$$
It is proved in \cite{CFFFR} that as projective varieties, $\mathcal{F}l_{n+1}^\mathbf{d}\cong\mathrm{Gr}_\mathbf{e}(M_\mathbf{d})$. To compare the poset structures, notice that for two points $\mathbf{d},\mathbf{d}'\in\mathcal{D}_{\mathrm{PBW}}$, when $I_\mathbf{d}\subseteq I_{\mathbf{d}'}$, the representation $M_{\mathbf{d}'}$ is a degeneration of $M_\mathbf{d}$.

For example, the point $\mathbf{p}$ lies in the interior of $\mathcal{D}_{\mathrm{PBW}}$, and the above isomorphism is proved by Cerulli Irelli, Feigin and Reineke in \cite{CFR}.

For $\mathbf{d}\in\mathcal{D}_{\mathrm{PBW}}$, the defining ideal of the projective variety $\mathcal{F}l_{n+1}^\mathbf{d}$ is described in \cite{CFFFR2}. Such an ideal is in fact the initial ideal of the Pl\"ucker ideal describing the complete flag variety with respect to some degree function. The linear map $\varphi$ in \cite[Section 5.3]{FFR2} (see also \cite[Section 7]{FFFM}) maps $\mathcal{D}_{\mathrm{PBW}}$ affinely to a face $\mathcal{F}$ of the tropical flag variety $\mathrm{trop}(\mathcal{F}l_{n+1})$ in the Pl\"ucker embedding (see \cite{FFR2} for details).

More geometric and representation-theoretic properties of the varieties $\mathcal{F}l_{n+1}^\mathbf{d}$ can be found in \cite[Section 5]{CFFFR}.
\end{example}

\begin{remark}
A point in the cone $\overline{\mathcal{D}}$ gives a filtration on any finite dimensional irreducible representation of the quantum group $U_\nu(\mathfrak{g}(Q))$. It is natural to ask for a basis of these representations, together with a nice polyhedral parametrization, which is compatible with such a filtration.
\end{remark}

\begin{example} \label{example:A3}
We consider $Q=1\longrightarrow 2\longleftarrow 3$. 
The AR-quiver has the form
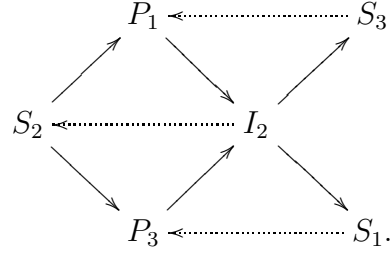
\centerline{
\xymatrix{
& P_1 \ar[dr] & & S_3 \ar@{.>}[ll] \\
S_2 \ar[ur] \ar[dr]& & I_2 \ar[ur] \ar[dr] \ar@{.>}[ll]&\\
& P_3 \ar[ur] & & S_1. \ar@{.>}[ll]
}
}

The three AR-conflations are 
$$(1).\ S_2 \rightarrowtail P_1\oplus P_3 \twoheadrightarrow I_2,\ \ (2).\  P_1 \rightarrowtail I_2\twoheadrightarrow S_3,\ \ (3).\  P_3 \rightarrowtail I_2\twoheadrightarrow S_1.$$
For a proper subset $I\subseteq \{1,2,3\}$ we denote $\cE_I$ the smallest exact structure on $\cA$ containing the AR-conflations indexed by $I$. We have $2^3 = 8$ different exact structures, cf. \cite[Section 4.2]{BHLR}.

By fixing a total ordering on $\ind{\cA}$, say 
$$M_1:=S_2,\ M_2:=P_1,\ M_3:=P_3,\ M_4:=I_2,\ M_5:=S_3,\ M_6:=S_1,$$
we identify $K_0^{\mathrm{add}}(\cA)$ with $\mathbb{Z}^6$. Let $x_1,\cdots,x_6$ be coordinates of $\mathbb{Z}^6$. Then $$K_0(\cE_1)=\mathbb{Z}^6/(x_1-x_2-x_3+x_4)$$ 
and 
$$K_0(\cE_{12})=\mathbb{Z}^6/(x_1-x_2-x_3+x_4,x_2-x_4+x_5).$$ 

The kernel of the canonical $\mathbb{Z}$-module map $K_0(\cE_1) \twoheadrightarrow K_0(\cE_{12})$ is given by $\mathbb{Z}(x_2-x_4+x_5)$, and $\cC^{\cE_{12},\cE_{1}}=\mathbb{R}_{\geq 0}(x_2-x_4+x_5)$. The dual cone 
$\mathcal{D}^{\cE_{12},\cE_{1}}$ consists of points $\mathbf{d}=(d_1,\cdots,d_6)\in\mathbb{R}^6$ such that $d_2+d_5>d_4$ and $d_1+d_4=d_2+d_3$. Its closure is a 5-dimensional cone in $\mathbb{R}^6$ having a 4-dimensional linearity space. 

We denote $\diamond_{12}$ and $\diamond_1$ the multiplications in the Hall algebra $\mathcal{H}(Q,\cE_{12})$ and $\mathcal{H}(Q,\cE_{1})$ respectively.

The Hall algebra $\mathcal{H}(Q,\ce_{12})$ is generated by $[M_i]$ for $i=1,\cdots,6$. For most of pairs $i, j$, $\Ext^1_{\ce_{12}}(M_i, M_j) = \Ext^1_{\ce_{12}}(M_j, M_i) = 0,$ so for such pair $i,j$, we have the following relations in $\mathcal{H}(Q,\ce_{12})$:

$$[M_i]  \diamond_{12} [M_j] = \frac{1}{|\Hom(M_i, M_j)|} [M_i \oplus M_j] = \frac{|\Hom(M_j, M_i)|}{|\Hom(M_i, M_j)|} [M_j]  \diamond_{12} [M_i].$$

Note that since $\ce_{12}$ is hereditary, we also have $\frac{1}{|\Hom(M_i, M_j)|} = \frac{1}{\left\langle M_i, M_j \right\rangle_{\ce_{12}}}$, whenever $\Ext^1_{\ce_{12}}(M_i, M_j) = 0.$

The only products of generators that are not given simply by the multiples of classes of their direct sums are the following:

$$[M_4] \diamond_{12} [M_1] = [M_1 \oplus M_4] + (q-1) [M_2 \oplus M_3],$$
$$[M_5] \diamond_{12} [M_2] = [M_2 \oplus M_5] + (q-1) [M_4].$$

Then all these identities, except for the last one, still hold in $\mathcal{H}(Q,\ce_1)$, and the last one is replaced by
$$[M_5] \diamond_1 [M_2] = [M_2 \oplus M_5].$$

Naturally, the commutation relations between the generators change accordingly. In $\mathcal{H}(Q,\ce_{12}),$ we have two non $q-$commuting pairs of generators $[M_i]$:
\begin{align} \label{M1:M4}
[M_4] \diamond_{12} [M_1]-q[M_1] \diamond_{12} [M_4] = (q-1) [M_2] \diamond_{12} [M_3],
\end{align}
\begin{align} \label{M2:M5}
[M_5] \diamond_{12} [M_2]-[M_2] \diamond_{12} [M_5]= (q-1) [M_4].
\end{align}

The coefficient $q$ ahead of $[M_1] \diamond_{12} [M_4]$ appeared because $|\Hom(M_1, M_4)| = q$, hence
$$[M_1] \diamond_{12} [M_4] = q^{-1} [M_1 \oplus M_4].$$

In $\mathcal{H}(Q,\ce_1),$ we still have 
$$[M_4] \diamond_1 [M_1]-q[M_1] \diamond_1 [M_4]= (q-1) [M_2] \diamond_1 [M_3],
$$
but the relation (\ref{M2:M5}) transforms to
$$[M_5] \diamond_1 [M_2]-[M_2] \diamond_1 [M_5]= 0.$$

Each point $\mathbf{d}\in\mathcal{D}^{\cE_{12},\cE_{1}}$ gives rise to a filtration on $\mathcal{H}(Q,\cE_{12})$ whose associated graded algebra is exactly $\mathcal{H}(Q,\cE_1)$.
\end{example}

\begin{example} \label{disjoint_A2}
Consider the disjoint union of two quivers of type ${\tt A_2}: \quad Q = 1 \longrightarrow 2 \quad 3 \longrightarrow 4.$

The AR-quiver has the form

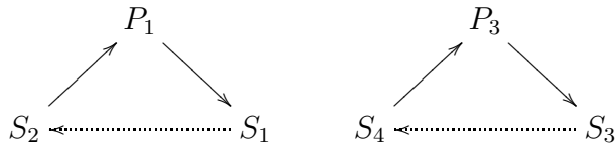
\centerline{
\xymatrix{
& P_1 \ar[dr] & & & P_3 \ar[dr] \\
S_2 \ar[ur] & & S_1 \ar@{.>}[ll]& S_4 \ar[ur] & & S_3 \ar@{.>}[ll]
}
}

There are two $\AR-$conflations:
$$S_2 \rightarrowtail P_1 \twoheadrightarrow S_1$$
and
$$S_4 \rightarrowtail P_3 \twoheadrightarrow S_3.$$

We have $2^2 = 4$ different exact structures.
In particular, there are $2$ different exact structures $\ce^1, \ce^2$ strictly between $\ce^{\add}$ and $\ce^{\max}$:
$\ce^1$ contains the first $\AR-$conflation but does not contain the second $\AR-$conflation;
$\ce^2$ contains the second $\AR-$conflation, but does not contain the first one.

Their Hall algebras are naturally isomorphic. The isomorphism is induced by the equivalence of exact categories
$$F: (\ca, \ce^1) \overset\sim\to (\ca, \ce^2),$$
$$S_1 \mapsto S_3, \qquad S_2 \mapsto S_4, \qquad P_1 \mapsto P_3,$$
$$S_3 \mapsto S_1, \qquad S_4 \mapsto S_2, \qquad P_3 \mapsto P_1.$$

Note that the isomorphism is not given by the identity on objects, and these $2$ exact structures correspond to different faces of the cone $\cC.$
\end{example}

\begin{example} \label{Geigle2}
Let $\Lambda$ be a hereditary Artin algebra of infinite representation type and $\mathscr{P}$ its full subcategory defined by the preprojective component, as in Example \ref{Geigle}.
Since $\mod \Lambda$ is $\Hom-$ and $\Ext^1-$finite, so is $\mathscr{P}$ endowed with an arbitrary exact structure $\cE.$ Our results show the existence of functions on $\Iso(\mathscr{P})$ that induce degenerations from $\cH(\cE^{\max})$ to $\cH(\cE)$ and from $\cH(\cE)$ to $\cH(\mathscr{P}, \cE^{\add}).$ Moreover, for any additive function on the the set $\Iso(\mathscr{P}),$ the associated graded to the induced filtration on $\cH(\cE^{\max})$ is the Hall algebra of a certain exact structure of $\mathscr{P}.$

Since the category has infinitely many indecomposables, one should consider the cones $\cC(\mathscr{P})$ and $\cd(\mathscr{P})$ in an infinite-dimensional vector space. For a pair of exact structures $\cE' < \cE$, the corresponding cones live in (mutually dual) finite-dimensional vector spaces whenever the set $\ind(\cP(\cE'))\backslash \ind(\cP(\cE))$ is finite.
\end{example}

\section{Remarks on extriangulated structures on an additive category} \label{extriangulated}

In this section, we will briefly discuss one direction in which we can generalize our results. This will be the subject of the upcoming sequel, where all the details will be given.

Nakaoka and Palu recently introduced a notion of \emph{extriangulated categories} \cite{NakaokaPalu1}. This is a unification of exact and triangulated categories. One can also consider it as an axiomatisation of full extension-closed subcategories of triangulated categories, although it is not known yet whether every extriangulated category admits an embedding (as a full extension-closed subcategory) into a triangulated category. The datum defining an extriangulated category consists of an additive category $\cA$, a bifunctor $\mathbb{E}: \cA^{\op} \times \cA \to \Ab$ and a so-called \emph{realization} correspondence $\mathfrak{s}$ that essentially defines the class of conflations. This datum must then satisfy certain axioms that simultaneously generalize axioms of exact and of triangulated categories. 

As we mentioned in Introduction, the associativity of Hall algebras of exact categories can be explained in terms of higher categories. Namely, to each exact category one can associate an $\infty-$category of special kind: an exact $\infty-$category. To each exact $\infty-$category, Dyckerhoff and Kapranov \cite{DK} associated a simplicial space, called the \emph{Waldhausen $S_{\bullet}-$space}. They proved that it satisfies a certain property that they called being \emph{ 2-Segal}. If we one starts from a $\Hom-$ and $\Ext^1-$finite exact category, the associativity of its Hall algebra can be seen as a shadow of this property of the corresponding $S_{\bullet}-$space. Stable $\infty-$categories are exact, and the associativity of derived Hall algebras of their homotopy categories is again just a shadow of their $S_{\bullet}-$spaces being 2-Segal. Very recently, Nakaoka and Palu \cite{NakaokaPalu2} proved that for every additive exact $\infty-$category (or, equivalently, for every exact $\infty-$category in terminology of Barwick \cite{Barwick}), its homotopy category admits a natural extriangulated structure. This suggests that one should be able to define associative Hall algebras of $\Hom-$ and $\Ext^1-$finite extriangulated categories, generalizing those of exact and of triangulated categories. In the upcoming sequel, we will give such a definition for all extriangulated categories satisfying certain natural finiteness conditions. Note that To\"en \cite{Toen} defined Hall algebras only for algebraic triangulated categories, but Xiao and Xu \cite{XX1} proved that the same definition applied for an arbitrary triangulated category (satisfying certain finiteness conditions) produces an associative algebra without any assumption on the existence of an enhancement. The same is true for extriangulated categories: their Hall algebras depend only on 1-categories, although in the presence of $\infty-$categorical enhancement, Hall algebras themselves admit, in some sense, higher categorical enhancements in the form of Waldhausen $S_{\bullet}-$spaces.

An additive category $\cA$ can be endowed with many extriangulated structures $(\mathbb{E}, \mathfrak{s})$. Unlike exact structures, extriangulated structures on an additive category do not form a lattice, and the reason is twofold. First, the definition of extriangulated structures suggests that they should probably form a cofiltered category, rather than just a poset. The precise structure has not been studied yet. Second, when an additive category admits a triangulated structure, the latter provides an extriangulated structure that has no non-zero projectives or injectives. Thus, whatever variation of the containment order we might consider, this structure would be maximal. But it is well-known that any category that admits a triangulated structure, admits at least two of them that differ from each other by the choice of a sign. More generally, there exist additive categories with several triangulations \cite{Balmer}. 
 
 From the point of view of Hall algebras, the second observation seems to be inconsequential. Indeed, in all the known situations all triangulated structures on a given additive category are related by global automorphisms of the category (that do not change the shift functor), see \cite{Balmer} for details. From the form of the structure constants, it follows that in this situation (derived) Hall algebras of different triangulated structures, whenever defined, are all isomorphic to each other. 
 
 It is not immediately clear whether the first observation has any relevance to Hall algebras either. Let $\cA$ be an additive category. Assume that each extriangulated structure on $\cA$ has a well-defined Hall algebra and that extriangulated structures on $\cA$ form a set. Consider the quotient of this set by the following equivalence relation: 
 $(\mathbb{E}, \mathfrak{s}) \sim (\mathbb{E'}, \mathfrak{s'})$
 if the identity map on objects in $\cA$ induces an isomorphism of the Hall algebras 
$\ch(\mathbb{E}, \mathfrak{s}) \overset\sim\to \ch(\mathbb{E'}, \mathfrak{s'}).$
We can endow this quotient with the following partial order: we say that 
$(\mathbb{E'}, \mathfrak{s'}) \leq (\mathbb{E}, \mathfrak{s})$
if $\ch(\mathbb{E'}, \mathfrak{s'})$ is a degeneration of $\ch(\mathbb{E}, \mathfrak{s})$.
It seems probable that this endows our quotient with a structure of a meet semilattice. Indeed, every exact structure on $\cA$ can be seen as an extriangulated category in a unique way. In particular, the split exact structure $(\cA, \cE^{\add})$ defines the extriangulated structure that we denote by $(\cA, \mathbb{E}^{add}, \oplus).$ Its Hall algebra as of an extriangulated category is the same as $\cH(\cE^{\add})$, i.e. a skew polynomial algebra. This is then the minimal element in our poset, the argument is similar to arguments in Section 4. 
The previous paragraph suggests that our quotient set of extriangulated structures might be a bounded complete lattice.

Naturally, we would like to define this poset without referring to Hall algebras. Under certain conditions, for an extriangulated structure $(\mathbb{E}, \mathfrak{s})$ on $\cA$, we can define and study intervals $[(\mathbb{E}^{add}, \oplus), (\mathbb{E}, \mathfrak{s})]$. Hu, Zhang and Zhou \cite{HZZ} recently generalized the notion of {\it proper classes} (originally due to Beligiannis \cite{Beligiannis}) from triangulated categories to extriangulated categories. Zhu and Zhuang \cite{ZZ} investigated the condition $\Ex = \AR$ and proved their theorems in a larger generality that we considered in the present paper - namely, in the setting of extriangulated categories. Ogawa \cite{Ogawa} considered defects and effaceable functors for extriangulated structures on additive categories $\cA$ with weak kernels. He proved that in this setting, these two notions again coincide and form a Serre subcategory of the category of finitely presented modules over $\cA.$ The converse statement is not yet known, i.e. it is not verified whether for an arbitrary $\ca$ with weak kernels there exists a Serre subcategory of its category of finitely presented modules such that each its Serre subcategory is the category of contravariant defects of conflations in an extriangulated structure on $\ca$. However, the notion of AR conflations is defined for extriangulated categories \cite{INP}, and their defects are always simple objects in the category of finitely presented modules.
Using all this, we can generalize our Theorem \ref{classif_general} as follows.

Let $\cA$ be a $\Hom-$finite, idempotent complete, $k-$linear locally finite small additive category endowed with an $\mathbb{E}-$finite extriangulated structure $(\mathbb{E}, \mathfrak{s})$. Then there are lattice isomorphisms between the lattices \footnote{The partial order in each of them is given by containment.} of 

\begin{itemize}
    \item Proper classes of extriangles in $(\mathbb{E}, \mathfrak{s})$;
    \item Relative extriangulated structures $\mathbb{E}_{\cd}$ with respect to full subcategories $\cd$ of $\cA$;
    \item  Relative extriangulated structures $\mathbb{E}^{\cd}$ with respect to full subcategories $\cd$ of $\cA$;
    \item Extriangulated structures $(\mathbb{E'}, \mathfrak{s'})$ such that $\Imm(\mathfrak{s}') \subset \Imm(\mathfrak{s})$ and $\mathbb{E'}$ is given by the restriction of $\mathbb{E}$ on the preimage of $\Imm(\mathfrak{s}')$;
    \item Subsets of the set of AR sequences in $(\mathbb{E}, \mathfrak{s})$;
    \item Subsets of the set of all the indecomposables in $\cA$ that are not projective in $(\mathbb{E}, \mathfrak{s})$;
    \item Subsets of the set of all the indecomposables in $\cA$ that are not injective in $(\mathbb{E}, \mathfrak{s})$.
\end{itemize}

In particular, each of these lattices is Boolean. If there are finitely many indecomposables in $\cA$ that are not projective in $(\mathbb{E}, \mathfrak{s})$, this lattice is a face lattice of a simplicial cone whose extremal rays are given by the generators of $AR(\mathbb{E}, \mathfrak{s})$.

Since we did not give any required definitions, we leave the proof to the upcoming sequel.

\end{document}